\documentclass[11pt,oneside]{amsart}
\usepackage[body={6in,9in}]{geometry}
\usepackage[backref]{hyperref}
\usepackage[alphabetic,backrefs,lite]{amsrefs}
\usepackage{amsmath}
\numberwithin{table}{section}
\numberwithin{figure}{section}
\usepackage{tikz-cd}
\usepackage{xypic}
\usepackage{subcaption}
\DeclareCaptionSubType*{figure}

\newcommand{\fa}{\mathfrak{a}}
\newcommand{\fb}{\mathfrak{b}}
\newcommand{\fd}{\mathfrak{d}}

\newcommand{\sA}{\mathcal{A}}
\newcommand{\sE}{\mathcal{E}}
\newcommand{\sI}{\mathcal{I}}
\newcommand{\sK}{\mathcal{K}}
\newcommand{\sL}{\mathcal{L}}
\newcommand{\sM}{\mathcal{M}}
\newcommand{\sQ}{\mathcal{Q}}

\newcommand{\sS}{\mathcal{S}}

\newcommand{\CC}{\mathbb{C}}
\newcommand{\FF}{\mathbb{F}}
\newcommand{\GG}{\mathbb{G}}
\newcommand{\PP}{\mathbb{P}}
\newcommand{\QQ}{\mathbb{Q}}
\newcommand{\ZZ}{\mathbb{Z}}
\newcommand{\kalg}{k^\mathrm{alg}}
\newcommand{\ksep}{k^\mathrm{sep}}

\newcommand{\bp}{\mathbf{p}}
\renewcommand{\P}{\mathbf{P}}

\newcommand{\Isym}{\sI_\mathrm{sym}}
\newcommand{\Ssym}{\sS_\mathrm{sym}}

\DeclareMathOperator{\Aut}{Aut}

\DeclareMathOperator{\Div}{Div}
\let\div\relax
\DeclareMathOperator{\div}{div}
\DeclareMathOperator{\End}{End}
\DeclareMathOperator{\Jac}{Jac}

\DeclareMathOperator{\Sp}{Sp}

\DeclareMathOperator{\GL}{GL}
\DeclareMathOperator{\Prym}{Prym}

\DeclareMathOperator{\Pic}{Pic}

\DeclareMathOperator{\res}{res}

\renewcommand{\H}{\operatorname{H}}
\renewcommand{\:}{\colon}

\newcommand{\grpC}{\mathrm{C}}
\newcommand{\grpD}{\mathrm{D}}
\newcommand{\grpV}{\mathrm{V}}

\newtheorem{theorem}{Theorem}[section]
\newtheorem{proposition}[theorem]{Proposition}
\newtheorem{corollary}[theorem]{Corollary}
\newtheorem{lemma}[theorem]{Lemma}

\theoremstyle{definition}
\newtheorem{definition}[theorem]{Definition}
\newtheorem{remark}[theorem]{Remark}
\newtheorem{example}[theorem]{Example}
\usepackage{thmtools, thm-restate}

\title{On $(2,2)$-decomposable genus $4$ Jacobians}
\hypersetup{
	pdfauthor   = {},
	pdftitle    = {On (2,2)-decomposable genus 4 Jacobians},
	pdfsubject  = {},
	pdfkeywords = {},
	backref=true, pagebackref=true, hyperindex=true, colorlinks=true,
	breaklinks=true, urlcolor=blue, linkcolor=blue, citecolor=blue,
	bookmarks=true, bookmarksopen=true}
\author{Nils Bruin}
\author{Avinash Kulkarni}
\thanks{The first author acknowledges the support of the Natural Sciences and Engineering Research Council of Canada (NSERC), funding reference number RGPIN-2018-04191.
The second author acknowledges support by
a Simons Collaboration Grant (550029).
}
\address{Department of Mathematics, Simon Fraser University, Burnaby, BC, Canada V5A 1S6}
\email{nbruin@sfu.ca}
\urladdr{http://www.cecm.sfu.ca/~nbruin}

\date{September 2, 2023}

\begin{document}
\begin{abstract}
We consider the question of when a Jacobian of a curve of genus $2g$ admits a $(2,2)$-isogeny to two polarized dimension $g$ abelian varieties. We find that one of them must be a Jacobian itself and, if the associated curve is hyperelliptic, so is the other.

For $g=2$ this allows us to describe $(2,2)$-decomposable genus $4$ Jacobians in terms of Prym varieties.
We describe the locus of such genus $4$ curves in terms of the geometry of the Igusa quartic threefold.
We also explain how our characterization relates to Prym varieties of unramified double covers of plane quartic curves, and  we describe this Prym map in terms of $6$ and $7$ points in $\PP^3$.
  		
We also investigate which genus $4$ Jacobians admit a $2$-isogeny to the square of a genus $2$ Jacobian and give a full description of the hyperelliptic ones. While most of the families we find are of the expected dimension $1$, we also find a family of unexpectedly high dimension~$2$.
\end{abstract}
\maketitle

\section{Introduction}

The motivation for our study is to describe genus four curves $X$ such that $\Jac(X)$ is decomposable in a particular way. For a principally polarized abelian variety $A$ of dimension $2g$, we say it is $(n,n)$-decomposable if it admits a polarized isogeny $\phi\colon J_1\times J_2 \to A$, where $J_1,J_2$ are principally polarized abelian varieties of dimension $g$ and such that the multiplication-by-$n$ map on $J_1\times J_2$ factors through $\phi$.

One can make $(n,n)$-decomposable abelian varieties by \emph{gluing} $J_1$ and $J_2$ along their $n$-torsion: suppose we have two principally polarized abelian varieties $J_1,J_2$ of dimension $g$, with an antisymplectic isomorphism $\tau\colon J_1[n]\to J_2[n]$. The graph $\Delta$ of this isomorphism in $J_1\times J_2$, which is isomorphic to $J_1[n]$ and $J_2[n]$ as a group scheme, is then a maximal isotropic subgroup of $(J_1\times J_2)[n]$ with respect to the product pairing Weil pairing. Hence, $A=(J_1\times J_2)/\Delta$ admits a principal polarization, and $A$ is $(n,n)$-decomposable by construction. 

For $g=1$, the study of $(n,n)$-decomposable abelian surfaces has a long history, for $n=2,3,4$ going back to Jacobi, Legendre \cite{legendre-traite}, Goursat \cite{Goursat1885}, and Bolza \cite{Bolza1887}; mainly in the language of reducing hyperelliptic integrals to elliptic ones. For more modern treatment, see \cite{Kuhn1988} for $n=3$,   \cite{BruinDoerksen2011} for $n=4$, and \cite{MagaardShaskaVolklein2009} for $n=5$.

For $g>1$, the picture becomes significantly more complicated. Principally polarized abelian varieties of dimension $4$ are not generally Jacobians and, if they are, not generally Jacobians of hyperelliptic curves. For $n=2$ we obtain a description that identifies the nature of decomposable Jacobians and how hyperellipticity affects them. The proof uses the Torelli theorem; see Section~\ref{S:proof_Theorem_1}. Throughout this paper we consider base fields of characteristic not equal to $2$.

\begin{theorem}\label{T:galchar}
Let $X$ be a curve of genus $2g$ such that $\Jac(X)$ is $(2,2)$-decomposable into factors $J_1, J_2$. Then, possibly after interchanging $J_1$ and $J_2$, the curve $X$ admits a double cover $X\to C_1$ of a genus $g$ curve such that $J_1=\Jac(C_1)$ and $J_2=\Prym(X,C_1)$.

If $C_1$ is hyperelliptic then the hyperelliptic quotient is $\PP^1$ and the tower $X\to C_1\to\PP^1$ expresses $X$ as a quartic cover.

\begin{enumerate}
\item \label{T:galchar:generic}
  If $X\to\PP^1$ is not geometrically Galois with Galois closure $X!\to X$ then $J_2\simeq \Jac(C_2)$ for some genus $g$ hyperelliptic curve $C_2$ that fit Diagram~\ref{fig:D4}.

\item \label{T:galchar:hyperelliptic}
  If $X$ is hyperelliptic then $X\to\PP^1$ is Galois and admits an intermediate cover $X\to C_2$ as in Diagram~\ref{fig:V4}, where $C_2$ is a hyperelliptic curve of genus $g$ and $\Jac(C_2)\simeq J_2$

\item \label{T:galchar:V4cover}
  If $X$ is not hyperelliptic and yet $X\to\PP^1$ is geometrically Galois then, possibly over a quadratic base field extension, we have curves $C_2,C_3$ as in Diagram~\ref{fig:V4nh} such that $J_2\simeq \Jac(C_2)\times \Jac(C_3)$.
\end{enumerate}
\end{theorem}
\begin{figure}[h]
	\centering
	\subcaptionbox{\label{fig:D4}}
{
	\begin{minipage}{0.3\textwidth}
		\[
		\begin{tikzcd}[ampersand replacement=\&]
			\& X! \ar[dl] \ar[d] \ar[dr] \\
			X \ar[d] \& Z \ar[d] \ar[dr] \ar[dl] \& C_2 \ar[d] \\
			C_1 \ar[dr] \& Y \ar[d] \& Q \ar[dl] \\
			\& \PP^1
		\end{tikzcd}
		\]
	\end{minipage}  
}
	\subcaptionbox{\label{fig:V4}}
	{
		\begin{minipage}{0.3\textwidth}
			\[
			\begin{tikzcd}[ampersand replacement=\&]
				\&X\arrow[dr]\arrow[d]\arrow[dl]\\
				C_1\arrow[dr]\& \PP^1\arrow[d] \& C_2\arrow[dl]\\
				\&\PP^1
			\end{tikzcd}
			\]
		\end{minipage}
	}
	\subcaptionbox{\label{fig:V4nh}}
{
	\begin{minipage}{0.3\textwidth}
		\[
		\begin{tikzcd}[ampersand replacement=\&]
			\&X\arrow[dr]\arrow[d]\arrow[dl]\\
			C_1\arrow[dr]\& C_2\arrow[d] \& C_3\arrow[dl]\\
			\&\PP^1
		\end{tikzcd}
		\]
	\end{minipage}
}
	\caption{Diagrams for Theorem~\ref{T:galchar}}
	\label{F:galchar}
\end{figure}
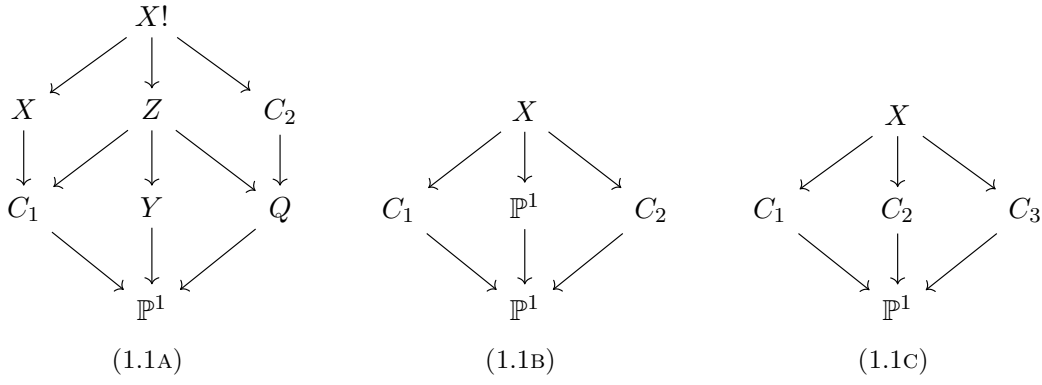

For $g=2$ and $n=2$, the curve $C_1$ is automatically hyperelliptic. Case~\ref{fig:V4nh} would have $C_2,C_3$ of genus $1$, so $X$ would be bielliptic. We ignore that case here, since there are other methods of designing genus $4$ Jacobians that decompose into elliptic curves. Since $1\equiv -1\pmod{2}$, we have that a symplectic (anti)isomorphism $\Jac(C_1)[2]\simeq \Delta \simeq\Jac(C_2)[2]$ corresponds to two points $a,b$ in the moduli space $\sM_2(\Delta)$ of genus $2$ curves with a full $2$-level structure $\Delta$ on their Jacobians (see Section~\ref{S:2level}). This moduli space has a quartic model $\sI^\Delta\subset\PP^4$, which we refer to as an \emph{Igusa quartic} (see Section~\ref{S:igusa}). Its moduli interpretation can be explicitly realized by the fact that the intersection with the tangent space $\sK_a=T_a \sI^\Delta \cap \sI^\Delta$ yields a Kummer quartic surface isomorphic to $\sK_{C_a}=\Jac(C_a)/\langle -1\rangle$. As observed by Van der Geer \cite{vanderGeer1982Siegel}, $\Jac(C_b)$ occurs as a Prym variety of a double cover $X\to C_a$ precisely when $b$ lies on this tangent space, expressed as $\P(a,b)=0$ (see Section~\ref{sec:notation}). As is mentioned in \cite{vanderGeer1982Siegel}, this can be proved over $\CC$ using theta function identities in \cite{Fay1973theta}. We give a synthetic proof of this fact that also holds in positive characteristics other than $2$ in Section~\ref{S:synthetic}. We summarize the conclusions concerning decomposable Jacobians of genus $4$ curves below.

\begin{theorem}\label{T:intro_construction}
Let $C_a,C_b$ be genus $2$ curves corresponding to distinct points $a,b\in \sI=\sI^{\Delta}$. Then $\Jac(C_b)=\Prym(X,D)$ for some quadratic twist $D$ of $C_a$ if and only if $\P(a,b)=0$. Suppose this is the case.
\begin{enumerate}
	\item If $\P(b,a)\neq 0$ then $X$ is not hyperelliptic and a singular quartic plane model of $D$ is $T_a\sI\cap T_b\sI\cap \sI$. The ramification locus of $X\to D$ is supported on the singularity of this model.
	\item If $\P(b,a)=0$ then $T_a\sI\cap T_b\sI\cap \sI$ is a double-counting conic passing through $a,b$ and five other points that are singularities of both $K_a$ and $K_b$.
\end{enumerate}
\end{theorem}

\begin{remark}
Propositions~\ref{P:polar_construction} and \ref{P:hyperelliptic_construction} give explicit constructions for Diagrams~\eqref{fig:D4} and \eqref{fig:V4} from points $a,b\in\sI$. We paraphrase them more informally here. In the case where $\P(b,a)\neq 0$ then inside $T_b\sI\simeq\PP^3$ we have two objects: We have a model $D_a=T_a\sI\cap T_b\sI\cap \sI$ of $C_a$ with a singularity at $b$ and we have the tangent cone of $\sK_b=T_b\sI\cap\sI$ at $b$, with six lines on it corresponding to the tropes of $\sK_b$ passing through $b$. Projection from $b$ yields a double cover $D_a\to L_a$, with $L_a$ a line, as well as a conic $Q_b$ with six marked points on it, both lying in $\PP^2$. We get a degree $2$ map $Q_b\to L_a$ by sending a point $p\in Q_b$ to $T_pQ_b\cap L_a$, i.e., the intersection of its tangent line to $Q_b$ with $L_a$. This yields $C_a\to\PP^1$ and $C_b\to Q_b\to \PP^1$ which allows the construction of \eqref{fig:D4}, where the twist chosen for $C_b$ corresponds to the twist of $X\to C_a$.

In the case where $\P(a,b)=0$, we have that $T_a\sI\cap T_b\sI$ is a trope to both $\sK_a=T_a\sI$ and $\sK_b=T_b\sI$ and we see that $L=T_a\sI\cap T_b\sI\cap \sI$ is a conic passing through $a,b$, as well as five singularities that are common to both $\sK_a$ and $\sK_b$. This marks exactly the ramification loci on a common genus $0$ curve $L$ to obtain a diagram of the form $\eqref{fig:V4}$.
\end{remark}

In Theorem~\ref{T:intro_construction} above, we identify $\Jac(C_b)$ as the Prym variety associated to a double cover of a singular plane section $D$ of $\sK_b$.
This can be seen as a degeneration of a more general construction: a sufficiently general plane section of $\sK_b$ yields a nonsingular plane quartic curve $D$ and the double cover $\Jac(C_b)\to \sK_b$ yields an unramified double cover $X\to D$ such that $\Jac(C_b)=\Prym(X,D)$.
Verra~\cite{Verra1987prym} identifies the fibre of the Prym map (i.e., the moduli space of unramified double covers $X\to D$ such that $\Prym(X,D)=\Jac(C_b)$) as birational to a quotient of the space of plane sections of $\sK_b$. Over a non-algebraically closed base field, not all such genus $3$ curves $D$ allow a degree $4$ model on $\sK_b$ and the natural question arises of which ones do.

The unramified double cover $\pi\colon X\to D$ implies there is a filtration of Galois modules
\[0\subset V_1 \subset V_5\subset \Jac(D)[2] \text{ with } V_5/V_1\simeq \Jac(C_b)[2],\]
where $V_1$ is the kernel of $\pi^*\colon \Jac(D)\to\Jac(X)$ and $V_5$ is the subspace that pairs trivially with $V_1$. We prove the following in Section~\ref{S:prymg3}.

\begin{proposition}
Let $C$ be a genus $2$ curve and suppose that $X\to D$ is an unramified double cover with $D$ of genus $3$ such that $\Jac(C)=\Prym(X,D)$. Then 
\begin{enumerate}
	\item $D$ admits a model as a plane section of $\sK_C$ if and only if $0\to V_1\to V_5\to \Jac(C)[2]\to 0$ is split, so that $\Jac(C)[2]$ is a direct symplectic summand of $\Jac(D)[2]$.
	\item $D$ admits a model as a plane section of the projective dual $\sK_C^\vee$ if and only if $D$ has even and odd theta characteristics $\theta_e,\theta_o$ over $k$ such that $V_1$ is generated by $\theta_e-\theta_o$.
\end{enumerate}
\end{proposition}

In Section~\ref{S:point_configuration} we consider the descriptions of $\sM_2(2)$ and $\sM_3(2)$ as birational to the configurations of $6$ and $7$ labelled points in $\PP^3$ respectively and find a surprising relation in terms of the Prym map between these two via the expression of an Igusa quartic as a symmetroid. We prove the following.

\begin{proposition}
Let $p_1,\ldots,p_7\in\PP^3$ be seven points in general position. Then there is a $2$-dimensional linear system of quadrics based at these points. There is a unique point $p_0$ that completes this to a Cayley octad, which determines a curve $D$ of genus $3$. Given $p_0$, the points $p_6,p_7$ determine a pair of odd theta characteristics $\theta_6,\theta_7$ on $D$. Let $X\to D$ be the double cover determined by the two-torsion class $\theta_6-\theta_7$.

Then $\Prym(X,D)=\Jac(C)$ for a genus $2$ curve $C$ that is a double cover of the rational normal curve through $p_0,\ldots,p_5$, ramified over those points.
\end{proposition}

Finally, in Section~\ref{S:M2} we consider genus $4$ curves $X$ such that $\Jac(X)$ is $(2,2)$-decomposable into $\Jac(C_a)\times\Jac(C_a)$, so that $M_2(\ZZ)\subset \End\Jac(X)$.
The key observation is that an isomorphism $C_a\to C_b$ establishes an additional symplectic isomorphism $\Jac(C_a)[2]\simeq\Jac(C_b)[2]$ which, together with the gluing identification $\Jac(C_a)[2]\simeq\Delta\simeq\Jac(C_b)[2]$, yields an automorphism $\sigma\in\Sp_4(\FF_2)\simeq S_6$ of $\Delta$, which acts linearly on $\sI^\Delta$.
This leads us to consider the loci $\P(a,\sigma(a))=0$,
the geometry of which depends on the conjugacy class of $\sigma$.

By Theorem~\ref{T:galchar} we see that for non-hyperelliptic $X$, we need a quadratic map $\mu\colon \PP^1\to\PP^1$ that maps the branch locus of $C_b\to\PP^1$ onto itself. Not all degree $6$ loci admit such a map, but the relation $\P(a,\sigma(a))=0$ allows us to describe and construct them; see Example~\ref{E:M2_nonhyp}.

In Section~\ref{S:M2_hyp} we consider the situation where $X$ is hyperelliptic, i.e., $\P(a,\sigma(a))=\P(\sigma(a),a)=0$.
For hyperelliptic $X$ we need an automorphism $\mu\colon \PP^1\to \PP^1$ such that the branch locus of $C_a$ and its image overlap in five points. We classify all families based on the conjugacy classes of $\sigma$ they induce; see Table~\ref{T:hyperellipticM2}.
Surprisingly, for one conjugacy class of $\sigma$ we find that the second relation is automatically satisfied and we find a two parameter family, which is of higher dimension than one would expect; see Example~\ref{E:2dim_hyplocus}.

\section{Notation and preliminaries} \label{sec:notation}

We work over a field $k$, not necessarily algebraically closed, of characteristic distinct from $2$. By a \emph{curve} we normally mean a geometrically integral projective one-dimensional variety defined over $k$. We do encounter singular models of curves, mainly as singular plane curves. For these, we will have a nonsingular curve at hand, together with a birational morphism to the singular model.·

\subsection{Polarity}\label{S:polar}

Let $X\subset\PP^n$ be a degree $d$ hypersurface, described by the vanishing of a homogeneous form $F$. We obtain a form of bidegree $(d-1,1)$ on $\PP^n\times \PP^n$ by considering
\[\P_F(a,b)=\nabla_a F\cdot b=\sum_{i=0}^n \partial_i F|_a\, b_i\]
We have $\P_F(a,b)=0$ precisely when $b$ lies in the tangent space to $X$ at $a$.
By fixing $b$ we get a degree $(d-1)$ form $\bp_b F=\nabla F\cdot b$. This and the iterates of the operator define the \emph{polar hypersurfaces} of $X$ at $b$,
\[
  \bp_{b,F}^{(j)}=\underbrace{\bp_b\circ\cdots\circ \bp_b}_{j \text{ times}}F.
\]
The degree of $\bp_{b, F}^{(j)}$ is $d-j$ and one can check that $\bp_{b, F}^{(d-1)}$ is the equation of the tangent space of $X$ at $b$.

\subsection{Prym varieties}

\begin{definition}
Given a double cover of curves $\pi\colon X\to C$, we write $\Prym(X,C)$ for the connected component containing the identity in the kernel of $\pi_*\colon \Jac(X)\to\Jac(C)$.
\end{definition}

Let $g$ be the genus of $C$ and let $2r$ be the degree of the branch locus of $\pi$. By Riemann-Hurwitz, we have that the genus of $X$ is $2g-1+r$ and hence that $\dim\Prym(X,C)=g-1+r$.
When $r=0,1$, the principal polarization on $\Jac(X)$ restricts to a principal polarization on $\Prym(X,C)$; see \cite{Mumford1974prym}.

When $\pi\colon X\to C$ is unramified, the map $\pi^*\colon \Jac(C)\to\Jac(X)$ has a kernel of order~$2$. When $\pi\colon X\to D$ is ramified, the map $\pi^*\colon \Jac(C)\to\Jac(X)$ is injective.
In either case, $\Prym(X,C)\cap \pi^*\Jac(C)\simeq \Prym(X,C)[2]$.

\subsection{$2$-Level structures}\label{S:2level}

We work over a field $k$ of characteristic different from $2$. Let $C$ be a curve of genus $2$ over $k$ and let $\Jac(C)[2]$ be its Jacobian. Then $\Delta=\Jac(C)[2]$ is a finite separated group scheme of degree $16$ and exponent $2$. It also comes equipped with a Weil pairing $\Delta\times\Delta\to\mu_2$, which is non-degenerate and alternating. Here we write $\mu_2$ for the $2$-torsion subgroupscheme of the multiplicative group scheme $\GG_m$. Over $\kalg$, we have that $\Delta$ is isomorphic to the constant group scheme $(\ZZ/2\ZZ)^{2g}$. Colloquially, by a full $2$-level structure one would mean a choice of symplectic basis, i.e., a specific choice of an isomorphism $(\ZZ/2\ZZ)^{2g}\simeq \Delta$ as symplectic modules.

We generalize this notion so that we can apply it to non-algebraically closed base fields.
Given $\Delta$ over $k$ as described, a full $2$-level structure $\Delta$ on $\Jac(C')$ is an isomorphism $\Delta\to\Jac(C')[2]$, compatible with the pairings.

If $J_1=\Jac(C_1)$ and $J_2=\Jac(C_2)$ both have the same full $2$-level structure then the (anti)diagonal embedding $\Delta\subset (J_1\times J_2)[2]$ has the property that the product pairing on $(J_1\times J_2)[2]$ restricts to the trivial pairing on $\Delta$. Hence, $\Delta$ is a maximal isotropic subgroup. It then follows by \cite{MilneAbVars1986}*{Proposition~16.8} that $A=(J_1\times J_2)/\Delta$ admits a principal polarization. The main motivating question for this paper is to determine when $A$ is isomorphic to a Jacobian of a curve $X$ of genus $4$ as principally polarized abelian variety. We denote this by $A\simeq \Jac(X)$.

\subsection{Genus $2$ curves and their quartic surfaces}
In this section we review the theory of genus $2$ curves and their Kummer surfaces. Let us first fix some terminology. We refer to a quartic surface of degree four in $\PP^3$ with $16$ nodal singularities as a \emph{ Kummer quartic surface}. For such a surface there are also $16$ planes called \emph{tropes} that intersect the quartic in a conic with multiplicity $2$.
Each singularity lies on six tropes and each trope passes through six singularities, forming a classic $(16)_6$ Kummer configuration. The projective dual of a Kummer quartic surface is again a Kummer quartic surface.

For a curve $C$ of genus $2$, both $\Pic^0(C/k)/\langle -1\rangle$ and $\Pic^1(C/k)/\langle \iota\rangle$ admit models as quartic Kummer surfaces that are projectively dual to each other. We review this construction below starting from the latter and in Proposition~\ref{P:description-gauss-map} we give a description of the Gauss map between the two. The ideas are largely classical and to some degree described in \cite{CasselsFlynn1996}*{Chapters 4, 5}. We include a relatively self-contained account here for the convenience of the reader.
 
We fix an affine model
\[C\colon y^2=f(x)=f_6x^6+\cdots+f_0 \text{ where } f\in k[x] \text{ is square-free of degree $5$ or $6$}.\]
The pull-back of $x=\infty$ on the $x$-line to $C$ yields a degree $2$ canonical divisor $\kappa$. The divisor $3\kappa$ is very ample and provides a degree $6$ embedding $\tilde{C}$ in $\PP^4$ with coordinates $(1:x:x^2:x^3:y)=(x_0:x_1:x_2:x_3:y_3)$. The image is the base locus of a linear system $\tilde{\sQ}$ spanned by the quadrics
\[
\begin{aligned}
\tilde{Q}_1&= x_1^2-x_0x_2,\\
\tilde{Q}_2&= x_0x_3-x_1x_2,\\
\tilde{Q}_3&= x_2^2-x_1x_3,\\
\tilde{Q}_4&= f_0x_0^2+f_1x_0x_1+f_2x_1^2+f_3x_1x_2+f_4x_2^2+f_5x_2x_3+f_6x_3^2-y_3^2.
\end{aligned}
\]
The restriction to $y_3=0$ yields a linear system $\sQ=\langle Q_1,Q_2,Q_3,Q_4\rangle$ of quadrics in $\PP^3$. The quadrics $Q_1,Q_2,Q_3$ define a rational normal curve $L$. Projection onto $(x_0:\cdots:x_3)$ expresses $\tilde{C}$ as a double cover of $L$, branched where $Q_4=0$.
It is a classical fact that the locus of singular quadrics in $\sQ$ yields a Kummer quartic surface
\[\sK_C^\vee\colon G=\det(\eta_1Q_1+\eta_2Q_2+\eta_3Q_3+\eta_4Q_4)=0.\]
This particular surface has a distinguished trope $\eta_4=0$.

We consider its ambient space as dual to $\PP^3$, with dual coordinates $(\xi_1:\ldots:\xi_4)$. Under the Gauss map $\gamma\colon (\PP^3)^\vee\dashrightarrow\PP^3$ defined by $p\mapsto \nabla_p(G)$, we have that $\sK_C=\gamma(\sK_C^\vee)$ is again a quartic Kummer surface with $16$ nodal singularities. The distinguished trope on $\sK_C^\vee$ corresponds to the distinguished singularity $(0:0:0:1)$ on $\sK_C$. This surface is isomorphic to the Kummer surface $\Jac(C)/\langle -1\rangle$.

Note that by Riemann-Roch, any degree $3$ divisor class $\fd\in \Pic^3(C)$ is represented by a $1$-dimensional linear system of effective divisors on $C$. The hyperelliptic involution $\iota \: C\to C$ corresponding to $(x,y)\mapsto(x,-y)$ induces an involution $\fd\mapsto 3\kappa-\fd$ on $\Pic^3(C)$.

\begin{proposition} With the notation above, we have a Galois-covariant bijection
	\[\Pic^3(C/\ksep)/\langle \iota \rangle\to \sK_C^\vee(\ksep)\]
Under this bijection, the plane section with $\eta_4=0$ corresponds to the image of divisor classes represented by $p+\kappa$, with $p$ a degree $1$ point on $C$.
\end{proposition}

\begin{proof}
Any degree $3$ divisor class $\fd$ admits an effective divisor $D$ with separated support. This support spans a plane $V_D$. Furthermore, since $D+\iota D$ is linearly equivalent to $3\kappa$, we see that $V_D\cup V_{\iota D}$ spans a hyperplane that intersects $\tilde{C}$ in $D+\iota D$.

The linear system $\tilde{\sQ}$ restricted to $V_D$ has a base locus of degree at least $3$ and therefore restricts to a system of dimension at most $2$ on $V_D$. Hence, there is a quadric $Q_D\in \tilde{\sQ}$ that contains $V_D$. The hyperelliptic involution has trivial action on $\tilde{\sQ}$, so the plane $V_{\iota D}$ is also contained in $Q_D$.
The hyperplane associated to the divisor $D + \iota D$ intersects $Q_D$ in a quadric of rank at most $2$, so $Q_D$ itself is of rank at most $4$, and hence singular.

Furthermore $\tilde{\sQ}$ contains no quadrics of rank less than $3$, since that would imply that $\tilde{C}$ lies in the union of two hyperplanes, which it does not.

We have that
\[\det(\eta_1\tilde{Q}_1+\eta_2\tilde{Q}_2+\eta_3\tilde{Q}_3+\eta_4\tilde{Q}_4)=\eta_4 G.\]
We show that $V_D$ must lie on a quadric $\tilde{Q}$ for which $G=0$. Suppose it does not. Then $\eta_4=0$, so $\tilde{Q}$ is a cone over a quadric $Q\in\langle Q_1,Q_2,Q_3\rangle$.
The rational normal curve $L$ is consequently contained in $Q$. 
If $Q$ is of rank $4$, then the two systems of lines on $Q$ intersect $L$ with multiplicity $2$ and $1$. But then the intersections $V_D \cap C,V_{\iota D} \cap C$ are pull-backs of divisors on $L$ and they have degree $4$ and $2$ respectively. But this contradicts that both $D$ and $\iota D$ of degree $3$ are contained in these intersections respectively.

The intersection of $\sK_C^\vee$ with $\eta_4=0$ consists of quadrics of rank $3$, with the singular locus intersecting $\tilde{C}$. This corresponds to divisors of the form $D=P+\kappa$, for which the complete linear system has a base point $P$. The plane $V_D$ intersects in $P+\iota P+\kappa$, and $V_D=V_{\iota D}$: in this case, $D$ and $\iota D$ can be distinguished by the base point.

Indeed, a quadric $Q$ in $\PP^4$ of rank $4$ has two rulings of planes on it, adding up to the hyperplane sections. These rulings cut out $1$-dimensional linear systems of divisors on $\tilde{C}$, adding up to $3\kappa$. 

For a quadric of rank $3$ there is just one ruling. As we saw above, for those with $\eta_4=0$, the planes actually cut out a $1$-dimensional linear system of degree $4$ divisors with a base locus of degree $2$: choosing one of the two base points gives us a linear system of degree $3$ divisors.

Combining the two cases, we see that the map $\{D, \iota D\} \mapsto Q_D$ defines the desired Galois covariant bijection.
\end{proof}

The bijection $\Pic^1(C/\ksep)\to \Pic^3(C/\ksep)$ induced by $\fd\mapsto \fd+\kappa$ yields the following.

\begin{corollary}\label{cor:pic1-kummer_dual}
We have $\Pic^1(C/\ksep)/\langle \iota\rangle \simeq \sK_C^\vee(\ksep)$ and the image of the natural embedding $C\to\Pic^1(C)$, consisting of the degree $1$ classes that admit an effective representative, lies in $\eta_4=0$.
\end{corollary}

For a degree $1$ class $\fd\in\Pic^1(C)$ we obtain an Abel-Jacobi map
\[j_\fd\colon C\to \Jac(C); \quad p\mapsto p-\fd.\]
When we compose this with $\Jac(C)\to \sK_C$, then we obtain a map that only depends on the class $\bar{\fd}$ modulo $\iota$. denoted by
\[\bar{\jmath}_{\bar{\fd}}\colon C\to \sK_C; \quad p\mapsto \overline{p-\fd}.\]

It is well-known that the image of $j_\fd$ gives a theta divisor in $\Jac(C)$. Over a non-algebraically closed field $k$, there may be no such divisors over $k$, but there is a divisor $2\Theta_C$ over $k$, for instance from the map $p\mapsto 2p-\kappa$. We have $|2\Theta_C|=\PP^3$ and this linear system yields the map $\Jac(C)\to \sK_C$. Since the inverse image of $\bar{\jmath}_{\bar{\fd}}$ is $j_\fd(C)+j_{\iota \fd}(C)$, which is linearly equivalent to $2\Theta_C$, we see that $\bar{\jmath}_{\bar{\fd}}(C)$ is a plane section $H_{\bar{\fd}}\cap \sK_C$.

We see that two distinct points $p,q\in C$ have the same image under $C\to \bar{\jmath}_{\bar{\fd}}(C)$ if $[p]-\fd=\fd-[q]$, i.e., if $2\fd=[p+q]$. Hence, if $2\fd\neq \kappa$ then the map is birational and the fibre over the singular point consists of the effective divisor representing $2\fd$. If $2\fd=\kappa$ then the map is two-to-one.

By Corollary~\ref{cor:pic1-kummer_dual} we can identify $\bar{\fd}$ with a point on $\sK_C^\vee$. This point is nonsingular unless $2\fd=\kappa$. For nonsingular $\bar\fd$, projective duality associates to $\bar\fd$ a tangent plane $T_a(\sK_C)$ with $a=\gamma(\bar{\fd})$.
The following well-known result (see \cite[pp. 435-436]{Verra1987prym} or \cite{Catanese2021kummer})) establishes that this plane agrees with $H_{\bar\fd}$. We include a proof here that avoids direct use of theta function identities.

\begin{proposition}\label{P:description-gauss-map}
Let $\bar\fd \in \sK_C^\vee$ be a nonsingular point.
Then
\[T_a(\sK_C)\cap \sK_C = \bar{\jmath}_{\bar{\fd}}(C)\text{ for } a=\gamma(\bar{\fd}).\]
\end{proposition}

\begin{proof}
Note that the result is stable under base field extension, so it is sufficient to check over algebraically closed base fields.
As discussed above, the map $C\to \bar\jmath_{\bar{\fd}}(C)$ has a singular image and the two points $p,q$ mapping to the singularity satisfy $2\fd=[p+q]$. Let $\fa=[p]-\fd$ and $-\fa=\fd-[q]$ be the two points on $j_\fd(C)$ that map to the singularity. Then $a=\bar{\fa}$ is the singularity of $\jmath_{\bar{\fd}}(C)$ and hence the image is indeed $T_a(\sK_C)\cap \sK_C$.

This means exactly that $[p]=\fd+\fa$ and $[q]=\fd-\fa$, i.e., that $\fd+\fa$ and $\fd-\fa$ admit effective representatives.
By Corollary~\ref{cor:pic1-kummer_dual} this means that
\begin{equation}\label{E:eta_characterization}
	\eta_4(\overline{\fd+\fa})=\eta_4(\overline{\fd-\fa})=0
\end{equation} The pair $\{\fa,-\fa\}$, and therefore $a$, is fully determined by this property.

Over an algebraically closed base field, $C$ has a rational Weierstrass point $\theta$ and we can change coordinates so that $f_6 = 0, f_5 = 1$ and $x(\theta) = \infty$. Then $\sK_C^\vee$ has a node at $(0:0:0:1)$, and the map $\bar{\fd}\mapsto \overline{{\fd}-\theta}$ defines an isomorphism $\sK_C^\vee \to \sK_C$ given by
$(\xi_1:\xi_2:\xi_3:\xi_4)=(\eta_4 : -\eta_3 : \eta_2 : - \eta_1)$.
Let $\fb=\fd-[\theta]$. Then \eqref{E:eta_characterization} becomes $\xi_1(\overline{\fb+\fa})=\xi_1(\overline{\fb-\fa})=0$. Let us write $b=\bar\fb$.

The group law on $\Jac(C)$ yields biquadratic forms $B_{ij}(a,b)$ (see \cite{Flynn1993}), characterized by the property that for  $\fa,\fb\in\Jac(C)$ and $a=\bar{\fa}$ and $b=\bar{\fb}$ we have
\[
    B_{ij}(a,b)=
    \xi_i(\fb+\fa)\xi_j(\fb-\fa)
    +\xi_i(\fb-\fa)\xi_j(\fb+\fa).
\]
These forms correspond to theta function relations upon identification of the coordinates $\xi_i$ with the appropriate theta functions, but the relations can be checked directly algebraically.

Now we take a generic point $\bar\fd$ on $\sK_C^\vee$ of the generic curve over $\QQ(f_0,f_1,f_2,f_3,f_4)$ and take $b=\overline{\fd-\theta}\in\sK_C$ via the isomorphism and $a=\gamma(\bar\fd)$.
Computation reveals that
$B_{1j}(a,b)=B_{i1}(a,b)=0$ for $i,j=1,\ldots,4$, which implies that \eqref{E:eta_characterization} holds.
See \cite{electronic_resources} for a transcript of such a check using a computer algebra system, but a very persistent reader could perform these computations by hand. It follows that for $a=\gamma(\bar\fd)$ we indeed get the desired plane section.

In fact, careful consideration shows this result holds over $\ZZ[\frac{1}{2}]$, so that the general result follows by specialization for any curve over a field of characteristic other than $2$. 
\end{proof}

\begin{corollary}\label{C:fod_obstruction}
If $\sK$ is a Kummer quartic surface over a field $k$ with a nonsingular $k$-rational point then $\sK=\sK_C$ for some genus $2$ curve $C$ over $k$.
\end{corollary}
\begin{proof}
First suppose that $a\in\sK(k)$ does not lie on a trope. Then Proposition~\ref{P:description-gauss-map} shows that $T_a(\sK)\cap \sK$ gives a singular model for $C$, defined over the base field. Otherwise, $T_a(\sK)\cap \sK$ gives a double-counting conic. But this conic has a rational point $a$, so it is isomorphic to $\PP^1$. The six nodes the conic passes through mark a locus $B$ and we get a model for $C$ by taking the a double cover of $\PP^1$ branched over $B$.
\end{proof}

\subsection{Igusa quartic threefolds}\label{S:igusa}

\setcounter{footnote}{1}
The \emph{Igusa quartic} or \emph{Castelnuovo--Richmond quartic}\footnote{See \cite{Dolgachev2012classical}*{p.~545} for a discussion on the naming.} is a quartic threefold in $\PP^5$ defined by
\[
\sI_4\colon\sum_{j=0}^5 x_j = 0, \qquad \left(\sum_{j=0}^5 x_j^2\right)^2 - 4\sum_{j=0}^5 x_j^4 = 0.
\]
Note that the first equation is linear, so $\sI_4$ is also a quartic hypersurface in $\PP^4$.

The singular locus of $\sI_4$ consists of $15$ lines corresponding to the different ways of splitting a set of six into three pairs, called \emph{synthemes}. These form an orbit under the permutation action of $S_6$ on the coordinates. One representative is cut out on $\sI_4$ by
\[x_0-x_1=x_2-x_3=x_4-x_5=0.\]
Another special locus, which we denote by $\sE$, consists of the orbit under $S_6$ of the hyperplane section $x_0+x_1+x_2=0$. Note that $x_3+x_4+x_5=0$ cuts out the same hyperplane section, so there are $10$ components to $\sE$.

The open part $\sI_4\setminus \sE$ is isomorphic to the moduli space $\sM_2(\Delta)$ of genus $2$ curves with full $2$-level structure $\Delta=(\ZZ/2\ZZ)^4$ on their Jacobians. 
This moduli interpretation is surprisingly concrete and, as we will describe below, applies to any quartic threefold $\sI$ that is isomorphic to $\sI_4$ over the algebraic closure of the base field $k$.
For a point $a\in \sI\setminus \sE$, the intersection $T_a\sI\cap \sI$ yields a quartic surface with $16$ singularities: $15$ from the intersection with the singular locus of $\sI$ and one distinguished singularity from the point of tangency $a$.
This is a Kummer quartic surface $\sK_a$. Indeed, it is no surprise we recover only $\sK_a$ and not a genus $2$ curve or its Jacobian directly: $2$-level structure is invariant under quadratic twists.

The space $\sM_2(\Delta)$ has automorphism group $\Sp_4(\FF_2)\simeq S_6$. Indeed, $\sI$ has an action of $S_6$ via permutation on the coordinates $x_0,\ldots,x_5$.

Suppose we have a curve $C_a$ such that $\sK_a=\Jac(C_a)/\langle-1\rangle$. Then $\Sp_4(\FF_2)$ acts through $S_6$ on the Weierstrass points of $C_a$ as well. The singularities on $\sK_a$ are the images of $\Jac(C_a)[2]$, so the $15$ singularities can be labelled by pairs of Weierstrass points of $C_a$.

\begin{remark}\label{R:outer_aut}
It follows that the permutation action of $\Sp_4(\FF_2)$ action on the coordinates, which corresponds to labelling the singularities with synthemes, is \emph{not} conjugate to the permutation action of $\Sp_4(\FF_2)$ on the Weierstrass points of $C_a$. The outer automorphisms of $S_6$ exchange the two representations.
\end{remark}

\begin{remark}
For nonsingular points in $\sE$ that lie in the component $x_0+x_1+x_2=0$ the tangent space is actually just $x_0+x_1+x_2=0$, which intersects $\sI$ in a double-counting quadric. These points correspond to abelian surfaces that are products of elliptic curves; see \cite{vanderGeer1982Siegel}.
\end{remark}

The possible $2$-level structures $\Delta$ for principally polarized abelian surfaces over $k$ are parameterized by the Galois cohomology set $\H^1(k,\Sp_4(\FF_2))$. Since the automorphism group of $\sI_4$ comes from a faithful representation $\Sp_4(\FF_2)\to \GL_6(k)$, we have that the corresponding twist $\sI^\Delta$ also admits a model inside a hyperplane in $\PP^5$, with again a singular locus consisting of $15$ lines. It follows that the moduli interpretation of $\sI^\Delta$ can be obtained in exactly the same way.

\begin{definition} Let $\sI\subset\PP^4$ be a quartic threefold over $k$ that is isomorphic to $\sI_4$ over $\ksep$. We say that $\sI$ is \emph{an Igusa quartic}. We write $\sE\subset \sI$ for the locus that gets mapped to the locus $\sE\subset\sI_4$ under an isomorphism. This is the locus where the tangent plane fails to cut out a normal quartic surface.
\end{definition}

The projective dual of $\sI_4$ is the Segre cubic threefold (see \cite{Dolgachev2012classical}*{Section~9.4.4}) which in $\PP^5$ can be described as
\[\sS_3\colon \sum_{i=0}^5 y_i =\sum_{i=0}^5 y_i^3=0.\]
Over an algebraically closed base field it can be characterized as the unique cubic threefold in $\PP^4$ with $10$ nodal singularities. For each twist $\sI^\Delta$ the projective dual is again a Segre cubic $\sS^\Delta$.

\section{Galois-theoretic characterization of curves with $2$-decomposable Jacobian}

\begin{lemma}\label{L:torelli_aut}
  Let $k$ be a field of characteristic different from $2$. Suppose $A_1,A_2$ are principally polarized abelian varieties of dimension $g$ and that $A_1[2]\simeq A_2[2]$ as $\Sp_{2g}(\FF_2)$-modules, with $\Delta\subset A_1\times A_2$ as graph. Suppose that $(A_1\times A_2)/\Delta\simeq \Jac(X)$ for a genus $2g$ curve $X$, as polarized abelian varieties.
  Then
  \begin{enumerate}
  \item[(a)] If $X$ is hyperelliptic then $X$ has two involutions $\tau_1,\tau_2$ and $\Jac(X/\tau_1)\simeq A_1$ and $\Jac(X/\tau_2)\simeq A_2$.
  \item[(b)] If $X$ is not hyperelliptic, then $X$ has an involution $\tau$ and
    $\Jac(X/\tau) \simeq A_1$ (up to relabelling).
  \end{enumerate}
  In either case, we have that $\Prym(X, X/\tau) \simeq A_2$.
\end{lemma}

\begin{proof}
First note that $A_1\times A_2$ has the involutions $[1]_{A_1}\times [-1]_{A_2}$ and $[-1]_{A_1}\times [1]_{A_2}$, which respect the polarization and restrict to the identity on $\Delta$, so these induce involutions on $\Jac(X)$ as well. The Torelli theorem (see for instance \cite{MilneJacobians1986}*{Theorem~12.1}) says that when $X$ is  hyperelliptic, both of these induce involutions on $X$, and on non-hyperelliptic $X$, exactly one of these induces an involution on $X$.

We check that if $[1]_{A_1}\times [-1]_{A_2}$ induces an involution $\tau$ on $X$, then $\Jac(X/\tau)\simeq A_1$.
Let $C := X/\tau$, and let $\pi\colon X \rightarrow C$ denote the quotient morphism. It is well-known
that
$\pi^*\pi_* = (1 + \tau) \in \End(\Jac(X))$.
From this and surjectivity of $\pi_*$ we have
\[
    \left(\Jac(X)^{(\tau=1)}\right)^\circ = \pi^*\Jac(C)
\]  
where ${}^\circ$ denotes taking the connected component of the identity. Let $\psi\colon A_1 \times A_2 \rightarrow \Jac(X)$ denote the isogeny with kernel $\Delta$, and notice that the intersection $\Delta \cap (A_1 \times \{0\})$
is trivial. In particular, the restriction $\psi\colon A_1 \times \{0\} \rightarrow \Jac(X)$ is
an injection whose image is both connected and fixed by $\tau$. Thus, $\dim \Jac(C) = \dim(A_1) = g$,
so it follows by Riemann-Hurwitz that $\pi$ is ramified (at two points). In particular, $\pi^*$ is
injective, so $A_1 \simeq \Jac(X/\tau)$.

Let us now prove that $\Prym(X, X/\tau) \simeq A_2$. The restriction $\psi\colon \{0\} \times A_2 \to \Jac(X)$ is an injection and $\tau$
acts by $[-1]$ on the image. In particular, it is contained in $\ker (1 + \tau)^\circ = \Prym(X, X/\tau)$. Comparing dimensions gives the result.
\end{proof}

\begin{remark}
  The converse also holds. Namely,
  if $X$ admits an involution $\tau$ with two fixed points,
  then $A_2 := \Prym(X, X/\tau)$
  is principally polarized, and putting $A_1 := \Jac(X/\tau)$, there are
  isomorphisms $A_1[2] \simeq A_2[2]$ and $(A_1 \times A_2)/\Delta \simeq \Jac(X)$.
  This result is due to Mumford~\cite{Mumford1974prym}.
\end{remark}

\subsection{Proof of Theorem~\ref{T:galchar}}\label{S:proof_Theorem_1}
The assumptions match those of Lemma~\ref{L:torelli_aut}, which gives us that $J_1=\Jac(C_1)$ and $J_2=\Prym(X,C_1)$. It remains to establish the additional characterizations of $J_2$ in case $C_1$ is hyperelliptic.

We start with case~(\ref{T:galchar:hyperelliptic}). Suppose that $X$ is hyperelliptic, say with hyperelliptic involution $\iota$. Then $X$ has three involutions, say $\tau, \tau\iota, \iota$, with $C_1=X/\tau$ and $C_2=X/(\tau\iota)$ of genus $g$, the hyperelliptic quotient $X/\iota$, and the degree four cover $X\to X/\langle \tau,\iota\rangle$ that factors through each of the other three. This yields the diamond diagram~\ref{fig:V4}. Furthermore, since $X$ is a hyperelliptic curve of even genus, we have that $X/\iota\simeq \PP^1$, and therefore $X/\langle \tau,\iota\rangle\simeq\PP^1$ as well. The maps $C_1\to\PP^1$ and $C_2\to \PP^1$ give an identification of $2g+1$ of the $2g+2$ Weierstrass points of each of $C_1,C_2$. This induces an isomorphism $\lambda\colon \Jac(C_1)[2]\to\Jac(C_2)[2]$. Writing $\phi\colon X\to C_1$ and $\psi\colon X\to C_2$, it remains to check that the sums of the pull-back maps $\phi^*+\psi^*\colon \Jac(C_1)\times \Jac(C_2)\to \Jac(X)$ is an isogeny, with kernel equal to the graph of $\lambda$. This check is straightforward.

For cases (\ref{T:galchar:generic}) and (\ref{T:galchar:V4cover})
we consider $X$ not hyperelliptic. We write $\tau$ for the involution of $X$ over $C_1$. Let $C_1\to L_1$ be the hyperelliptic cover.  We consider the tower of double covers $X\to C_1\to L_1$. From Riemann-Hurwitz we see that the ramification locus of $X$ over $C_1$ is of degree $2$.

We classify the possibilities for the geometric Galois group of the quartic cover $X\to L_1$: it is one of $\grpV_4$, $\grpC_4$, or $\grpD_4$, where the last one occurs exactly when $X\to L_1$ is not Galois itself. We will see that $\grpV_4$ corresponds to case~(\ref{T:galchar:V4cover}) and $\grpD_4$ corresponds to case~(\ref{T:galchar:generic}), while $\grpC_4$ cannot actually occur.

First suppose that $\Aut_{\ksep}(X,L_1)=\grpV_4$. Then, possibly over a quadratic base field extension, $X\to L_1$ admits three subcovers, as pictured in Diagram~\ref{fig:V4nh}. Since $X$ is assumed to be non-hyperelliptic, the genera of all three curves $C_1,C_2,C_3$ are positive.
It is straightforward to check that $\Prym(X,C_1)\simeq \Jac(C_2)\times \Jac(C_3)$. It remains to verify that $L_1\simeq \PP^1$. We have that $g(C_1)=g(C_2)+g(C_3)$, so at least one genus is even. If $C_2,C_3$ are both defined over $k$ then this implies $L_1$ is covered by an even genus hyperelliptic curve and hence $L_1\simeq \PP^1$. If $C_2,C_3$ are not defined over $k$, then they are quadratic conjugate and therefore $g(C_2)=g(C_3)$. It follows that $g(C_1)$ is even, so we obtain the same result.

Now suppose, for the purpose of contradiction, that $\Aut_{\ksep}(X,L_1)=\grpC_4$. In that case, $C_1\to L_1$ is the unique quadratic subcover. Let us write $R\in\Div(C_1)$ for the branch locus of $X$ over $C$. By Riemann-Hurwitz, $R$ is an effective, separated, degree $2$ divisor on $C_1$. Let $\iota$ be the hyperelliptic involution on $C_1$. We must have $\iota R=R$.

If $R$ is simply the pull-back of a degree $1$ divisor on $L_1$, we would have $\Aut_{\ksep}(X,L_1)\simeq \grpV_4$, so $R$ is the sum of two distinct branch points of $C_1/L_1$. In this case, $R$ is not equivalent to the fibre class of $C_1 \rightarrow \PP^1$ as $g(C_1) \geq 1$. 
We have that $X$ is obtained by adjoining the square root of a function $f$ with divisor $\div(f)=R-2\alpha$, where $\alpha$ is a divisor of degree $1$, but not necessarily an effective one. Under our assumptions, adjoining the square roots of $f$ or $\iota(f)$ yield the same result, so $\iota(f)f$ is a square. Its divisor is $2R-2(\alpha + \iota \alpha)$, so the question is whether $R-(\alpha+\iota \alpha)$ is principal. However, note that $\alpha+\iota \alpha$
lies in the fibre class, so $R - (\alpha+\iota \alpha)$ is not a principal divisor. Hence we see that a cyclic order $4$ Galois group cannot occur.
  
We are left with the remaining situation where $X\to L_1$ is not Galois. Since it is a tower of quadratic extensions, this means its Galois closure $X!$ is of degree $8$ with dihedral Galois group. We name the curves and maps as follows.
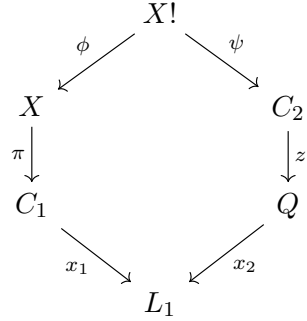
\[
  \begin{tikzcd}
    & X! \arrow[dl,"\phi" ' ] \arrow[dr,"\psi"] \\
    X \arrow[d,"\pi" ' ] && C_2 \arrow[d,"z"] \\
    C_1\arrow[dr,"x_1" ' ] && Q \arrow[dl,"x_2"] \\
    & L_1
  \end{tikzcd}
\]
Under the Galois correspondence, $X$ and $C_2$ correspond to the conjugacy classes of non-normal order $2$ subgroups. Tracking ramification yields that $g(C_2)=g(C_1)$ and that $Q,L_1$ are of genus $0$. Since $L_1$ has a genus $0$ double cover, we must have $L_1\simeq\PP^1$.

In order to check that $\Jac(X)$ indeed decomposes as expected, we observe that the covers $x_1\colon C_1\to L$ and $x_1\circ z\colon C_2\to L$ give an identification between the Weierstrass points of $C_1$ and $C_2$ and therefore induce an isomorphism $\lambda\colon \Jac(C_1)[2]\to \Jac(C_2)[2]$.

Through pull-back and push-forward of divisor classes we obtain a morphism
\[
  \pi^* + \phi_*\psi^*\: \Jac(C_1) \times \Jac(C_2) \rightarrow \Jac(X).
\]
We claim this is a polarized isogeny whose kernel is the graph of $\lambda$ and whose dual, composed with self-duality of domain and codomain, is
\[
  (\pi_*, \psi_*\phi^*)\: \Jac(X) \rightarrow \Jac(C_1) \times \Jac(C_2).
\]
First we check that the composition $\mu=(\pi_*, \psi_*\phi^*)\circ(\pi^* + \phi_*\psi^*)$ is multiplication-by-two. Indeed, it is straightforward to check that $\pi_*\phi_*\psi^*=x_1^*(x_2)_*z_*$. That map factors through $\Pic^0(Q)=\Pic^0(L)=0$, so it is constant $0$. Similarly, we have $\psi_*\phi^*\pi^*=z^*x_2^*(x_1)_*=0$. Hence, we see $\mu=(\pi_*\pi^*,\psi_*\phi^*\phi_*\psi^*)=(2,2)$.

We now determine the kernel of $\pi^* + \phi_*\psi^*$. Note that $\ker(\pi^* + \phi_*\psi^*) \subseteq \Jac(C_1)[2] \times \Jac(C_2)[2]$. 
It is easy to see that if $P$ is a Weierstrass point of $C_2$, then $\phi_*\psi^*(P)$ is the fibre over the corresponding Weierstrass point
of $C_1$ (considered as a divisor of degree $2$). In particular, $\phi_*\psi^*(\Jac(C_2)[2]) = \pi^*(\Jac(C_1)[2])$, and the kernel is indeed
the graph of $\lambda$. This completes the proof of Theorem~\ref{T:galchar}.

\begin{remark}
Geometrically, the data in Figure~\ref{fig:V4} is determined by a degree $2g+1$ locus on $\PP^1$ together with two points $a,b\in\PP^1$. Then $C_1,C_2$ are the double covers of $\PP^1$ ramified over $B\cup\{a\}$ and $B\cup\{b\}$ respectively, and $Q$ is the double over of $\PP^1$ ramified over $\{a,b\}$.

If we choose a coordinate on $\PP^1$ such that $a=\infty$, we get models
 \[\begin{aligned}
 	C_1\colon& y_1^2=d_1\,f(x),\\
 	C_2\colon& y_2^2=d_2\,(x-b)\,f(x),\\
 	Q\colon & y_0^2=d_1d_2\,(x-b),\\
 	X\colon & y_1^2=d_1\, f\Big(\frac{1}{d_1d_2}y_0^2+b\Big)
\end{aligned}\]
where $f$ is a square-free polynomial of degree $2g+1$ and $b\in k$ with $f(b)\neq 0$. The factors $d_1,d_2\in k^\times$ represent classes in $k^\times/k^{\times2}$.
\end{remark}

\begin{remark}
The data in Figure~\ref{fig:D4} is determined by a genus $g$ hyperelliptic curve $C_1$ together with a degree $1$ divisor $\alpha$ on $C_1$ such that $2\alpha$ admits an effective representative. If $2\alpha$ is hyperelliptic we are back to Figure~\ref{fig:V4}, where $X$ is hyperelliptic.

Alternatively, we can also specify a hyperelliptic curve $C_2$ as a double cover of a genus~$0$ curve $Q$, together with a double cover $Q\to\PP^1$ with such that the image of the branch locus of $C_2\to Q$ remains separated.
This latter description shows that the relevant moduli space is actually rational.
\end{remark}

\begin{remark}
By tracing ramification in Figure~\ref{fig:D4} we find that in the general case, the genera of $Y,Z,X!$ are $g+1,2g+1,4g+1$ respectively. These genera drop with overlap between the branch locus of $C_2\to Q$ and the ramification locus of $Q\to\PP^1$.
Equivalently, this corresponds with an overlap between the branch locus of $X\to C_1$ and the Weierstrass points of $C_1$.

Let us take $Q\simeq \PP^1$ with affine coordinate $x$ and double cover to $\PP^1$ given by $x=u^2$. For $C_2\colon v^2=h(u)$ and $f(x)=\res_u(h(u),u^2-x)$ we obtain a model $C_1\colon y^2=f(x)$.

If $f$ is of degree $2g+2$ and $f(0)\neq 0$, then we find that $Y\colon w^2=xf(x)$ is of genus $g+1$. If $C_2\to\PP^1$ is ramified over $x=\infty$ then $f$ is of degree $2g+1$; in this case, $Y$ admits a model of the same form, but now has genus $g$.

If $f(x)=x \tilde{f}(x)$ then we find $Y\colon w^2 =\tilde{f}(x)$ has genus $g-1$.
Note that in the latter case, the degree $2$ effective divisor on $C_2$ supported at $x=0,\infty$ is twice the degree $1$ divisor $\alpha$. This means that $\Jac(C_1)$ has a rational $4$-torsion point.
\end{remark}

\section{The Prym locus on Igusa quartics} \label{S:synthetic}

Let $\sI$ be an Igusa quartic. For a point $a\in \sI\setminus\sE$, the moduli interpretation, in the form of a Kummer surface, can be obtained by $\sK_a=T_a\sI\cap\sI$. As is observed in \cite{vanderGeer1982Siegel}, the fact that $\sK_a$ occurs as a subvariety of $\sI$ suggests that $\sK_a\subset \sI$ has a moduli interpretation as well. This is indeed  the case: for a point $b\in \sK_a\subset \sI$, we have that $\Jac(C_b)$ is the Prym variety of a double cover $X_{a,b}$ of $C_a$, branched over the degree-$2$ divisor marked by the point $b$ on $\sK_a$. In \cite{vanderGeer1982Siegel}, this is proved using a theta function identity, verified in \cite{Fay1973theta}.

The relevance for us is that a pair of points $a,b\in \sI$ determines a pair of abelian surfaces $\Jac(C_a),\Jac(C_b)$ with compatible $2$-level structure and hence an isomorphism between the $2$-torsion subgroups with graph, say $\Delta$. Hence, there is a $(2,2)$-decomposable abelian fourfold $(\Jac(C_a)\times\Jac(C_b))/\Delta$.
By Theorem~\ref{T:galchar}, this is exactly a Jacobian if one is a Prym variety for a cover of the other curve. In the notation of Section~\ref{S:polar}, this holds if $\P(a,b)=0$.

We give a synthetic construction for $X_{a,b}$ given such $a,b$. This yields the result directly in arbitrary characteristics other than $2$. It also provides insight in possible obstructions to realizing this moduli interpretation over algebraically non-closed base fields.

Let $a,b\in \sI$ with $\P(a,b)=0$. Then $b$ lies on the tangent space of $\sI$ at $a$ and hence lies on $\sK_a=\sI\cap T_a\sI$. The tangent plane to $\sK_a$ at $b$ yields a point on the projective dual $\sK_a^\vee$, which by Corollary~\ref{cor:pic1-kummer_dual} represents $\Pic^1(C_a)/\langle \iota\rangle$, where $\iota$ is the involution on divisor classes induced by the hyperelliptic involution on $C_a$. By adjusting the quadratic twist of $C_a$ that we consider, we can make sure that $b$ lifts to a pair of degree $1$ divisor classes on $C_a$, say $\fd$ and $\iota(\fd)$.

By Proposition~\ref{P:description-gauss-map}, the images of the Abel-Jacobi maps $j_\fd(C_a),j_{\iota\fd}(C_a)\subset \Jac(C_a)$ map to the tangent plane section $D_a=T_b\sK_a\cap\sK_a$, which provides a singular quartic plane model of $C_a$.
Using that both $\sK_a$ and $\sK_b$ occur as hyperplane sections of $\sI$, we see that
\[D_a=T_a\sI\cap T_b\sI\cap \sI=\sK_a\cap\sK_b\]
occurs as a section of $\sK_b$, by a plane passing through $b$. This configuration allows us to explicitly construct a diagram as in Figure~\ref{fig:D4} and Theorem~\ref{T:galchar}.
In what follows we use the notation from Section~\ref{S:polar}.

\begin{proposition}\label{P:kummer-polar}
  Let $\sK_b\subset \PP^3$ be a quartic Kummer surface with distinguished node $b$. Denote the quartic form defining $\sK_b$ by $G$.
  \begin{enumerate}
  \item $\bp_{b,G}^{(2)}=0$  describes the tangent cone of $\sK_b$ at $b$.
  \item The locus $\bp_{b,G}^{(1)}=0$ intersected with $\sK_b$ consists of the six conics cut out by the six tropes on $\sK_b$ passing through $b$.
  \item The locus $\bp_{b,G}^{(1)}=\bp_{b,G}^{(2)}=0$ consists of six lines: each line is the locus where a relevant trope through $b$ is tangent to the cone.
  \item Under the projection from $b$, the image of $\bp_{b,G}^{(2)}$ is a plane conic $Q_b$ and the image of
    $\bp_{b,G}^{(1)}=\bp_{b,G}^{(2)}=0$ is a degree $6$ locus $B$ such that, for any curve $C_b$ with
    $\sK_b\simeq \Jac(C_b)/\langle-1\rangle$, we have that $C_b$ is a double cover of $Q_b$ branched at $B$.
    (The curve $C_b$ is unique up to quadratic twists.)
  \end{enumerate}
\end{proposition}

\begin{proof}
  \begin{enumerate}
  \item: This is immediate from the definition of the second polar.
  \item: A quartic Kummer surface has $16$ nodes and $16$ tropes. Each trope is a plane that intersects $\sK_b$ in a double-counting conic. It passes through six nodes, and conversely, every node lies on six tropes. A trope is a tangent plane to $\sK_b$ at any point of intersection with $\sK_b$.
    
    It follows that a conic on $\sK_b$ cut out by a trope passing through $b$ must lie in $\bp_{b,G}^{(1)}\cap \sK_b$. The intersection is of degree $4\cdot 3$, so six conics make up the entire locus.
  \item: Each trope $T$ passing through $b$ is tangent to the cone at a ray through $b$.
    That ray is the tangent line at $b$ to the conic $T\cap \sK_b$, so it is part of the indicated locus. Degree considerations show there are no other components.
  \item: The projective dual $\sK_b^\vee$ is again a quartic Kummer surface. Under duality, tropes correspond to nodes on the dual.
    Thus the trope of $\sK_b^\vee$ dual to the node $b$ passes through the six nodes dual to the tropes passing through $b$. Under duality, this corresponds to $Q_b$ together with a degree $6$ locus $B$.
    
    Over an algebraically closed field, we have that $\sK_b=\Jac(C_b)/\langle -1\rangle$ for some curve $C_b$ of genus $2$. For the dual we have $\sK_b^\vee=\Pic^1(C_b)/\langle \iota\rangle$, where $\iota$ is the morphism that hyperelliptic involution induces on divisor classes, and $Q_b\subset \sK_b$ is the image of the embedding $C_b\subset \Pic^1(C_b)$, where the nodes are the images of the Weierstrass points, i.e., the branch points of the hyperelliptic cover $C_b\to Q_b$. This proves the statement. \qedhere
  \end{enumerate}	
\end{proof}

\begin{remark}
Proposition~\ref{P:kummer-polar}(3) yields a very direct way of realizing the moduli interpretation in the form of a degree $6$ locus on a curve of genus $0$: If $\sI$ is defined by $F=0$ and $b$ is a point on $\sI$ then $\bp^{(1)}_{b,F}=\bp^{(2)}_{b,F}=\bp^{(3)}_{b,F}=0$ projected away from $b$ yields the intersection of a plane cubic and conic.
\end{remark}

\begin{remark}
From this construction one can also see in a fairly elementary way that the isomorphism class of the locus $B$ only depends on the twist $\sI$: the fifteen singularities on the tangent planes arise from the 15 lines that make up the singular locus of $\sI$. The six tropes through the distinguished node correspond to the odd theta characteristics of $C_b$. We know that the divisor classes of order $2$ on $C_b$ can be expressed uniquely as a difference of two distinct odd theta characteristics. Hence we can label each odd theta characteristic with the cardinality $5$ set of two-torsion classes that have the given characteristic in their representation. But that means we can label them with $5$-tuples of components of the singular locus of $\sI$, which is discrete data. Since $\sI\setminus \sE$ is connected we see that this labelling must be constant.
\end{remark}

\begin{proposition}\label{P:polar_construction}
  Suppose $\P(a,b)=0$ and $\P(b,a)\neq 0$, and that $a,b\notin \sE$.
  \begin{enumerate}
  \item The curve model $D_a=T_a(\sI)\cap T_b(\sI)\cap \sI$ is a quartic plane section of $\sK_a$ with a singularity at $b$. We have that $\sK_a\simeq \Jac(D_a)/\langle -1\rangle$
  \item Projection from the point $b$ yields a double cover $D_a\dashrightarrow L_a\simeq\PP^1$ defined outside of~$b$. The ramification points of this cover lie in the locus defined by $\bp^{(1)}_{b}$, i.e., they are the points of intersection of $D_a$ with the tropes of $\sK_b$ that pass through $b$.
  \item If $b$ admits a genus $2$ curve $C_b$, then $C_a\to\PP^1$ and $C_b\to Q_b\to\PP^1$ fit in Diagram~\ref{fig:D4}, so there is a double cover $X\to C_a$ such that $\Prym(X,C_a)=\Jac(C_b)$.
  \end{enumerate}
\end{proposition}

\begin{proof}
  \begin{enumerate}
  \item: From Proposition~\ref{P:description-gauss-map}, we know that plane sections of $\sK_a$ are
    of this type.
  \item: Since $D_a$ is a plane quartic curve with a singularity at $b$, we see that projection from $b$ indeed
    yields a degree $2$ map $D_a\to\PP^1$.
    Outside of $b$, this cover is ramified where the tangent line to $D_a$ passes through $b$, i.e., points where $\bp^{(1)}_b$ vanishes.
    Since $D_a\subset \sK_b$, Proposition~\ref{P:kummer-polar}(2) shows this happens on the tropes passing through $b$.
  \item: We use Proposition~\ref{P:kummer-polar}(4) to get a description of $C_b$ as a double cover of $Q_b$, with
    branch locus $B$ marked by the six tropes passing through $b$. We use that $D_a$ lies on $\sK_b$. Thus, projection
    from $b$ puts the line $L_a$ (defined in part~(2)) and the conic $Q_b$ in the same plane, with $L_a \cap Q_b$ being the image of the
    tangent cone of $D_a$ at $b$.
    The images of the six tropes to $\sK_b$ through $b$ yield tangent lines to $Q_b$; each tangent at a point in $B$.
    Each of those lines also passes through a branch point of $D_a\dashrightarrow L_a$.
    
    We define the degree $2$ map $Q_b\to L_a$ by sending a point $p$ on $Q_b$ to $T_p(Q_b)\cap L_a$. If $p\in B$,
    then the tangent line to $Q_b$ is the image of the corresponding trope. The intersection of that trope with
    $D_a$ is a ramification point, so this map sends $B$ onto the branch locus of $D_a\to L_a$.
    
    We see that the branch locus of $Q_b\to L_a$ consists of $L_a\cap Q_b$, which corresponds to the tangent cone
    of $D_a$ at $b$.
    \qedhere
  \end{enumerate}
\end{proof}

\begin{remark}
  The condition in Proposition~\ref{P:polar_construction}(3) regarding whether $b$ admits a genus $2$ curve $C_b$ is entirely arithmetic; over an algebraically closed field
  it is automatically fulfilled. The problem here is that while for any Igusa quartic $\sI$, one can construct for a sufficiently general point $a$ on $\sI$ a Kummer quartic surface $\sK_b=\sI\cap T_b\sI$, the construction from $Proposition~\ref{P:kummer-polar}(3)$ may lead to a degree $6$ locus on a nonsingular conic that is not isomorphic to $\PP^1$. In that case, we have the moduli of a Galois-invariant isomorphism class class of genus $2$ curves that does not contain a Galois-invariant representative. The fact that the field of definition of genus $2$ moduli may differ from the field of definition of any representative is well-known.
  
  Note that by Corollary~\ref{C:fod_obstruction} this problem does not arise for $\sK_a$, because $b$ provides a nonsingular point on it. The obstruction for $\sK_b$ is visible on $C_a$ in the form that while $b$ specifies a Galois-invariant divisor class on $C_a$, this class may fail to contain a Galois-invariant representative divisor. In that case, the cover $X\to C_a$ cannot be defined over $k$ either.
  
  For $\sI=\sI_4$, we see that the six lines in $Proposition~\ref{P:kummer-polar}(3)$ are defined over the ground field, so the conic they project to has rational points. Hence, none of these obstructions occur in this case.
\end{remark}

\begin{remark}
	In the proof of Proposition~\ref{P:polar_construction}(3) it perhaps seems more natural to obtain the double cover $Q_b\to\PP^1$ by projection from a point in $\PP^2$ not on $Q_b$. Indeed, the chord $L_a$ intersects in $Q_b$ in two points. The tangent lines at those points intersect in a point $y$ not on $Q_b$. Projection from $y$ yields a double cover $Q_b\to \PP^1$ that is ramified at $Q_b\cap L_a$.
\end{remark}

As we remarked, if $\P(a,b)=\P(b,a)=0$ then $X$ covers both $C_1$ and $C_2$ and $X$ is hyperelliptic. We conclude this section by showing how in that case the data determining Figure~\ref{fig:V4} can be obtained from an Igusa quartic.

\begin{proposition}\label{P:hyperelliptic_construction}
Suppose $a,b\in \sI\setminus \sE$ are distinct points and that $\P(a,b)=\P(b,a)=0$. Then $T_a\sI\cap T_b\sI$ is a trope to both $\sK_a$ and $\sK_b$. The intersection $\sK_a\cap\sK_b$ is a double-counting conic through $a$, $b$, and five common singularities of $\sK_a,\sK_b$.
\end{proposition}

\begin{proof}
Writing $F$ for the defining equation of $\sI$. The condition $\P(a,b)=\P(b,a)=0$ implies that $a$ lies on $\bp^{(1)}_{b,F}=\bp^{(3)}_{b,F}=0$. But Proposition~\ref{P:kummer-polar}(2) identifies this as one of the six tropes of $\sK_b$ through $b$. By symmetry we see this is also one of the six tropes of $\sK_a$ through $a$. Intersecting with $\sI$ hence yields a double-counting conic that lies on both $\sK_a$ and $\sK_b$. On $\sK_a$, the conic passes through six singularities -- specifically, the point $a$ as well as the five other intersection points of $T_a\sI$ with the singular locus of $\sI$. By symmetry, those five are also singularities on $\sK_b$. 
\end{proof}

\section{Prym varieties of curves of genus $3$}\label{S:prymg3}

The description in the previous chapter describes all the ways in which $\Jac(C_b)$ can arise as a Prym variety of a ramified double cover $X\to C_a$ and that the curve $C_a$ can be recovered as plane sections of $\sK_b$ passing through the distinguished node.

A general plane section of $\sK_b$ yields a quartic plane curve $D$, i.e., a curve of genus $3$, and pull-back under the double cover $\Jac(C_b)\to\sK_b$ yields an unramified double cover $X\to D$ and that $\Prym(X,D)=\Jac(C_b)$. As Verra \cite{Verra1987prym} describes, these models can be understood in terms of Abel-Prym maps. Over a non-algebraically closed base field, not all curves $D$ for which $\Prym(X,D)=\Jac(C_b)$ can be obtained as a plane section of $\sK_b$. With a little bit of Galois cohomology we can describe the situation.

Let $\pi\colon X\to D$ be an unramified double cover of a non-hyperelliptic curve $D$ of genus $3$ and let $\iota$ be the corresponding fixed-point-free involution on $X$.
The kernel of $\pi_*$ is not connected, but consists of two disjoint components that we denote by $\ker(\pi_*) = \Prym(X, D) \sqcup \ker(\pi_*)'$. 
There is a natural map $X\to \ker(\pi_*)$ by $P\mapsto [P-\iota(P)]$. The image does \emph{not} lie in $\Prym(X,D)$ (see \cite{bruin:prym}*{Remark~6.3}), but in the other component. That component is a principal homogeneous space of the Prym variety, so its isomorphism class corresponds to a class in the Galois cohomology group $\H^1(k,\Prym(X,D))$. It is easy to see it is of period dividing two, so writing $\Delta=\Prym(X,D)[2]$, the class is the image of some $\xi\in\H^1(k,\Delta)$. The involution $\iota_*$ acts on $\ker(\pi_*)'$ and the quotient is some twist $\sK^{(\xi)}$ of the Kummer surface of $\Prym(X,D)$. The embedding $X\subset\ker(\pi_*)'$ yields a model $D\subset \sK^{(\xi)}$ which, if $\sK^{(\xi)}$ admits a quartic model in $\PP^3$, is a plane section.

When $\Prym(X,D)=\Jac(C)$ for some curve $C$ of genus $2$, the Kummer surface $\sK_C$ admits a quartic model in $\PP^3$. There are two twists of particular interest: the surface $\sK_C$ itself and its dual $\sK_C^\vee$.

The twist $\sK^{(\xi)}$ can be understood in terms of the relation between $\Delta$ and $\Jac(D)[2]$, so we are led to consider various subgroups of the automorphism group $\Sp_6(\FF_2)$ of $\Jac(D)[2]$. A computer algebra package like GAP \cite{GAP4} or Magma \cite{magma} can easily compute the conjugacy classes of the subgroups and verify the various facts and characterizations of subgroups that we need. See \cite{electronic_resources} for a transcript.

The unramified double cover $X\to D$ marks a $2$-torsion point on $\Jac(D)[2]$ and the self-duality implies a filtration of group schemes
\[0\subset V_1 \subset V_5\subset\Jac(D)[2],\text{ with }\Delta\simeq V_5/V_1.\]
We need to consider the subgroup $H\subset\Sp_6(\FF_2)$ of the automorphism group that stabilizes such a decomposition. This is the point stabilizer of a single non-trivial $2$-torsion point, so it must be a subgroup of index $63$, and there is a unique conjugacy class of those. Indeed, $H$ acts on $V_5/V_1$ through the full $\Sp_4(\FF_2)$, with kernel an elementary $2$-group of order $32$. The action of $H$ on the $28$ bitangents marks orbits of lengths $12$ and $16$.
Viewing the odd theta characteristics of $D$ as quadratic forms on $\Jac(D)[2]$, the orbit of size $16$ is the collection $\{\theta : \theta(\epsilon) = 1 \in \mathbb{F}_2\}$,
where $\epsilon$ is the kernel of $\pi^*\: \Jac(D)[2] \to \Jac(X)[2]$.

If $\sK^{(\xi)}$ admits a quartic model in $\PP^3$ then it has $16$ \emph{tropes} that intersect in double-counting conics, with $\Delta$ acting simply transitively on them.
In other words, the collection of tropes $\Delta'$ is a torsor for $\Delta$, and we have the equality of classes $[\Delta'] = \xi$ in $\H^1(k, \Delta)$. 
On a plane section $D$ of $\sK^{(\xi)}$ these tropes cut out bitangents. These must be the bitangents in the orbit of length $16$ under the stabilizer $H$.
Conversely, the Galois action on the Abel-Prym image of $D \to \sK^{(\xi)}$ determines a Galois action on the bitangents, and therefore on $\Delta'$. In particular, the Galois
action on $\Jac(D)[2]$ determines $\xi$.
On the specific quartic surfaces $\sK_C$ and $\sK_C^\vee$ we see that the possible group actions are further restricted, as we discuss below.

On $\sK_C$ the tropes are partitioned in a set of six that pass through the distinguished node and $10$ that do not. Indeed, there is an index $16$ subgroup $H'\subset H$ that stabilizes such a partition. If Galois acts through $H$ then $\sK^{(\xi)}$ has a Galois-stable node and hence $\xi$ is trivial. The conjugacy class of $H'$ can also be characterized in a different way: it is the class of the stabilizer of a direct sum decomposition
\[\Jac(D)[2]\simeq V_2\times\Prym(X,D)[2],\]
where $V_2$ admits a filtration $0\subset V_1\subset V_2$.

Turning to the dual $\sK_C^\vee$, we have a unique distinguished trope, dual to the distinguished node on $\sK_C$. This leads us to the stabilizer $H''\subset H$ of a bitangent in the length $16$ orbit, which is also of index $16$ in $H$. In fact, we have $H'\simeq H''\simeq (\ZZ/2)\times \Sp_4(\FF_2)$, but they are non-conjugate in $\Sp_6(\FF_2)$.
Conversely, if $D$ has a distinguished bitangent, then there is a distinguished hyperplane section of $\sK^{(\xi)}$, so in particular $\sK^{(\xi)} \subset \PP^3$.

The group $H''$ also arises as the intersection of the stabilizers of two theta characteristics, one odd and one even, which are groups whose conjugacy classes are uniquely determined by their indices $28$ and $36$ respectively.
We summarize the characterizations above in the proposition below.
\begin{proposition}\label{P:kummer-plane-sections}
Let $\pi\colon X\to D$ be an unramified double cover with $D$ a non-hyperelliptic curve of genus $3$. Let $C$ be a curve of genus $2$ such that $\Jac(C)=\Prym(X,D)$.
\begin{enumerate}
\item[(a)] A model of $D$ can be obtained as a plane section of $\sK_C$ if and only if $\Jac(C)[2]$ occurs as a direct symplectic summand of $\Jac(D)[2]$.
\item[(b)] A model of $D$ can be obtained as a plane section of $\sK_C^\vee$ if and only if $D$ has an even and odd theta characteristic $\theta_e,\theta_o$ over the ground field such that $\theta_e-\theta_o$ generates the kernel of $\pi^*\colon \Jac(D)\to \Jac(X)$.
\end{enumerate}
\end{proposition}

\begin{corollary}\label{C:I_plane_sections}
Let $\sI^\Delta$ be an Igusa quartic with projective dual a Segre cubic $\sS^\Delta$. Let $W\subset \PP^4$ be a plane. Suppose that $D=W\cap\sI^\Delta$ is a smooth plane quartic curve and that the intersection $L=\sS^\Delta\cap W^\vee$ of the line $W^\vee$ dual to $W$ with the dual $\sS^\Delta$ of $\sI^\Delta$ is reduced. Then
\[\Jac(D)[2]=E[2]\times\Delta,\]
where $E[2]$ is the exponent $2$ group scheme whose non-trivial part is isomorphic to $L$ as a $0$-dimensional variety.
\end{corollary}
\begin{proof}
The curve $D$ already lies on $\sI^\Delta$, so for any point $b\in\sI^\Delta$ such that $D$ lies on $T_b\sI^\Delta$, we have that $D=W\cap\sK_b$ and hence a plane section. But those tangent planes correspond exactly to $L$, which is the intersection of a line $W^\vee$ with a cubic hypersurface $\sS^\Delta$, so we find three such points $b_1,b_2,b_3$ over the algebraic closure.

Hence, over $k(b_i)$, we see that $D$ occurs as a plane section of $\sK_{b_i}$. Even though $\sK_{b_i}$ is a quartic Kummer surface with a distinguished node, one may need to make a quadratic extension in order to find an abelian surface $A$ such that $\sK_{b_i}=A/\langle-1\rangle$. However, over such an extension, Proposition~\ref{P:kummer-plane-sections} yields that $\Delta$ is a direct summand of $\Jac(D)[2]$, at least over $k(b_i)$. Symmetry in $b_i$ and Galois-invariance yields that $\Delta$ is also a direct summand over the base field.

Note that marking a cyclic subgroup in the cofactor $E[2]$ amounts to choosing one of the three non-identity elements in it, which must pair up with our choice of $b_i$. This will yield a Galois-stable $2$-torsion class on $D$. If this class does not admit a representing divisor over $k(b_i)$, then the corresponding cover $X\to D$ cannot be realized over the base field either, explaining how $\sK_{b_i}$ may fail to be a Kummer surface of an abelian surface $\Prym(X,D)$ over its field of definition.
\end{proof}

\begin{remark}
We explain how the situation of Theorem~\ref{T:intro_construction} arises as a limit case of the general construction discussed in this section.
Corollary~\ref{C:I_plane_sections} shows that a plane quartic section $D=W\cap \sI$ generally lies in $T_b\sI$ for three different points $b$, giving rise to three Prym varieties of $D$. This corresponds to the three different choices of filtration $0\subset V_1\subset V_2$ for a fixed isotropic $V_2\subset \Jac(D)[2]$.

We write $\gamma\colon \sI\to\sS$ for the Gauss map that sends a point on $\sI$ to its tangent space (as a point in dual space). The map is birational, so outside of the singular loci on either side we have $\P_\sI(a,b)=0$ if and only if $\P_\sS(\gamma(b),\gamma(a))$. Hence, if $\P_\sI(a,b)=0$ then $\gamma(a)\in T_{\gamma(b)} \sS$. In particular, the line $W^\vee$ dual to the plane $W=T_a\sI\cap T_b\sI$ is tangent to $\gamma(b)$ and hence $W^\vee\cap \sS=\gamma(a)+2\gamma(b)$.

If we now consider a family of points $a_t$ such that $b\in T_{a_0}\sI$ then we get a family of curves $D_t=T_{a_t}\sI\cap T_b\sI\cap\sI$ where the general member is a genus $3$ curve $D_t$ together with an unramified double cover $X_t\to D_t$ and $\Prym(X_t,D_t)=\Jac(C_b)$. For $t=0$ we have a ramified double cover $X_0\to D_0$ such that $\Prym(X_0,D_0)=\Jac(C_b)$.
\end{remark}

\section{The Igusa quartic as a symmetroid}\label{S:point_configuration}
A particularly interesting picture emerges if we restrict to $\Delta$ for which $\sK_b$ and $\sK_b^\vee$ are isomorphic, i.e., such that $C_b$ has a rational Weierstrass point. In that situation, $\sI^\Delta$ can be expressed as a \emph{symmetroid}: as the locus of singular members of a linear system of quadrics or, equivalently, as a determinant of a symmetric matrix of linear forms.
 
We shall fix points $p_1, p_2, \ldots, p_5 \in \PP^3$ in general position. For ease of exposition we take them individually defined over the base field and assign them coordinates
\[
p_1 := ( 1 : 0 : 0 : 0 ), \
p_2 := ( 0 : 1 : 0 : 0 ), \
p_3 := ( 0 : 0 : 1 : 0 ), \
p_4 := ( 0 : 0 : 0 : 1 ), \
p_5 := ( 1 : 1 : 1 : 1 ).
\]
This yields a symmetroid expression for the classical Igusa quartic $\sI_4$,
but we only need the points to be Galois-stable as a set.
We denote by $\mathcal{L}$ the linear system of
quadrics in $\PP^3$ vanishing at $p_1, \ldots, p_5$. It is of projective dimension $9-5=4$ and is spanned by
\[
y_0 y_1 - y_2 y_3, \
y_0 y_2 - y_2 y_3, \
y_0 y_3 - y_2 y_3, \
y_1 y_2 - y_2 y_3, \
y_1 y_3 - y_2 y_3.
\]
We write $x_0,\ldots,x_4$ for coordinates relative to this basis and set $x_5=-x_0-\cdots-x_4$. Then the Gram matrix of a member of $\sL$ with coordinates $(x_0:\cdots:x_4)$ is
\[
\label{E:igusa_determinantal}
\mathcal{A}(x) :=
\begin{bmatrix}
	0 & x_0 & x_1 & x_2 \\
	x_0 & 0 & x_3 & x_4 \\
	x_1 & x_3 & 0 & x_5 \\
	x_2 & x_4 & x_5 & 0 \\
\end{bmatrix}
\]
To a point $p\in \PP^3$ with $p\notin\{p_1,\ldots,p_5\}$ we can associate the linear subsystem  $\sL_p\subset \sL$ of quadrics that vanish at $p$. 
This yields the rational map $\psi\colon \PP^3\dashrightarrow \sL^\vee$, defined outside $p_1,\ldots,p_5$.
\[
\psi(y_0, \ldots, y_3) =
(
y_0 y_1 - y_2 y_3 : 
y_0 y_2 - y_2 y_3 : 
y_0 y_3 - y_2 y_3 : 
y_1 y_2 - y_2 y_3 : 
y_1 y_3 - y_2 y_3
).
\]

\begin{proposition} \label{prop:determinantal-igusa-model}
	The image $\Ssym$ of $\psi$ is isomorphic to the Segre cubic
	$\mathcal{S}_3$.
	As a consequence, $\sI_4$ is isomorphic to the quartic hypersurface defined by
	\[\Isym\colon \det(\sA)=0.\]
\end{proposition}

\begin{proof}
	It is straightforward to verify that the image of $\psi$ is defined by the equation
	\[
	x_0 x_2 x_3 - x_1 x_2 x_3 - x_0 x_1 x_4 + x_1 x_2 x_4 + x_1 x_3 x_4 - x_2 x_3 x_4.
	\]
	This is a cubic threefold in $\PP^4$ with $10$ nodal singularities, defined over the base field. This property characterizes the Segre cubic.
	The variety $\Isym$ is defined by a symmetric determinantal equation, so the image of $\psi$ is the dual of $X$. Since the Segre cubic and Igusa quartic are projectively dual, we find that $\Isym\simeq \sI_4$.
\end{proof}

For a pair of sufficiently general points $p_6,p_7\in\PP^3$ we consider the linear system $\sL_{p_6,p_7}=(\sL_{p_6}\cap \sL_{p_7})\subset \sL$. Its base locus is a complete intersection of three quadrics in $\PP^3$, so it contains an eighth point $p_8$. Indeed, the line spanned by $\psi(p_6),\psi(p_7)$ intersects $\sS$ in a third point, which is $\psi(p_8)$. Thus the points $\{p_1,\ldots,p_8\}\subset\PP^3$ form the base locus of a net of quadrics $\sL_{p_7,p_8}$. Such a configuration of points is classically known as a \emph{Cayley octad}. A Cayley octad determines a curve of genus $3$ with an even theta characteristic. The curve can be obtained as the locus of singular quadrics in the net. If the points are labelled, they determine $2$-level structure on the curve. In our case, the distinguished role of $\{p_6,p_7\}$ marks a $2$-torsion class on the curve.

\begin{proposition} \label{prop:5+2+1}
	With the notation above, let $p_6,p_7\in\PP^3$ be points such that $p_1,\ldots,p_7$ are in general position. Let $p_8$ be the eighth point in the base locus of $\sL_{p_6,p_7}$ and let $D=\sL_{p_6,p_7}\cap \Isym$ be the locus of singular quadrics.
	Then the hyperplane $\sL_{p_8}$ contains $D$ and is tangent to $\Isym$.
\end{proposition}

\begin{proof}
	Since $\psi(p_8)\in \Ssym$, we have that $\sL_{p_8}$ is tangent to $\Isym$ by duality.
	Furthermore, since $\sL_{p_6,p_7}$ has $p_8$ in its base locus, we have $\sL_{p_6,p_7}\subset\sL_{p_8}$.
\end{proof}

\begin{remark}
	Note that a sufficiently general plane section $D=W\cap \Isym$ is a smooth plane quartic. The plane $W\subset \sL$ yields a line in $\sL^\vee$, which generally has three intersection points $p_6,p_7,p_8\in \Ssym$. We see that $D$ lies in three tangent hyperplanes to $\Isym$: $\sL_{p_6},\sL_{p_7},\sL_{p_8}$ and the intersection of any pair of them yields $W$. By the discussion above, we see that $p_1,\ldots,p_8$ forms a Cayley octad.
\end{remark}

\begin{corollary}
	Let $p_1, \ldots, p_5, p_6, p_7, p_8$ be the points of a Cayley octad. Then the plane quartic
	$D := \Isym \cap \sL_{p_6, p_7}$ is contained in the Kummer surface $\sL_{p_8} \cap \Isym$.
	
	In other words, the genus $3$ curve $D$ determined by the Cayley octad $p_1,\ldots,p_8$ with the $2$-torsion point represented by the differences of the theta characteristics marked by $p_6,p_7$ yield an unramified double cover $X\to D$ such that $\Prym(X,D)=\Jac(C)$, where $C$ is a genus $2$ curve that is a double cover of the rational normal curve through $p_1,\ldots,p_5,p_8$, and ramification locus at those points.
\end{corollary}

\section{Constructing genus $4$ Jacobians isogenous to a square}
\label{S:M2}

We return to the construction of Section~\ref{S:synthetic}, with an Igusa quartic $\sI$ and points $a,b\in\sI$ with $\P(a,b)=0$.
If we take $C_a\simeq C_b$ then $\Jac(X)$ is (2,2)-isogenous to $\Jac(C_a)^2$, and hence has a copy of $M_2(\ZZ)$ in its endomorphism ring.

There are several ways of enforcing such an isomorphism, and we can track them in the following way. If $C_a\simeq C_b$ then we have two ways of identifying $\Jac(C_a)[2]$ with $\Jac(C_b)[2]$: we have the isomorphism $\Jac(C_b)\to\Jac(C_a)$ restricted to the $2$-torsion and we have the identification $\Jac(C_a)[2]\simeq \Jac(C_b)[2]$ used for the gluing.
The composition yields an automorphism $\sigma\in\Sp_4(\FF_2)\simeq S_6$. The group $S_6$ acts by permutation on the coordinates $\sI\subset\PP^5$ although its permutation action is not conjugate to the action on the odd theta characteristics (see Remark~\ref{R:outer_aut}).
For each of the finitely many choices for $\sigma$ we can consider the corresponding locus $\P(a,\sigma(a))=0$ on~$\sI$. The geometry of this locus depends on the conjugacy class of $\sigma$, i.e., the cycle type as a permutation on six elements. We stick with $\sI=\sI_4$ here.

Note that if $\sigma(a)=a$, then the induced isogeny is just $\Jac(C_a)^2\to\Jac(C_a)^2$
defined by $(\fd_1,\fd_2)\mapsto(\fd_1+\fd_2,\fd_1-\fd_2)$, and hence we do not get a codomain that is the Jacobian of a genus $4$ curve. We also discard the elliptic locus $\sE$: there are other ways of constructing genus $4$ curves $X$ with a Jacobian isogenous to a product of elliptic curves.

In what follows we list representatives of the conjugacy classes as they act on the coordinates of $\sI$, so with $\sigma=(0,1,2)$ we mean the automorphism
\[\sigma\colon (x_0:x_1:x_2:x_3:x_4:x_5)\mapsto (x_1:x_2:x_0:x_3:x_4:x_5).\]
For $\sigma= (0), (0,1), (0,1)(2,3)(4,5)$ we find that the locus $\P(a,\sigma(a))=0$ on $\sI$ is entirely supported on the fixed and elliptic loci. For $\sigma=(0,1)(2,3)$ we find two irreducible surfaces, interchanged by $(4,5)$; see Example~\ref{E:2dim_hyplocus} for a parametrization. For the other conjugacy classes, after removing the elliptic loci, we find a single irreducible surface with nonsingular $\QQ$-rational points; see \cite{electronic_resources} for a transcript of our computations. Note that the existence of a nonsingular $\QQ$-rational point on an irreducible variety implies that the variety is also irreducible over the algebraic closure.

\begin{example}\label{E:M2_nonhyp}
As a fairly representative example, let us take $\sigma=(0,1,2)$ and the two points $a=(-55: -29:  49:  36:  20: -21)$ and $b=\sigma(a)=(-29 : 49 : -55 : 36 : 20 : -21)$. We follow the procedure outlined in Proposition~\ref{P:polar_construction}. We choose the twist of $C_b$ that makes it agree with $C_a$ and after changing coordinates to reduce coefficients and such that the models for $C_a$ and $C_b$ are identical, we find
\[C_b=C_a\colon y^2=2(x-7)x(x+2)(x+1)(x-4)\text{ with the map }\mu\colon x\mapsto \frac{7x^2 - 37x - 84}{x^2 + 11x}.\]
The map $\mu\colon \PP^1\to\PP^1$ is branched over $121x^2 + 478x + 3721=0$ and one can check that it maps $-2\mapsto -1\mapsto -4\mapsto -2$ and $\infty\mapsto 7\mapsto 0\mapsto\infty$, i.e., the map $\mu$ induces a permutation on the branch points of $C_b$ of cycle type $(1,2,3)(4,5,6)$.
We now have
\[\begin{tikzcd}
&&C_b\arrow{d}\\
C_a\arrow{dr}&&\PP^1\arrow{dl}{\mu}\\
&\PP^1
\end{tikzcd}\]
which we can complete with the fibre product $X!=C_b\times_{\PP^1} C_a$ and the double covers $X!\to X\to C_a$ that are induced. We find
\[
    X\colon \begin{cases}
    y^2=2(x-7)x(x+2)(x+1)(x-4)\\
    z^2=-9828y - 5577x^3 - 13914x^2 + 131121x + 125538
    \end{cases}.
\]
It follows that $\Jac(X)$ is isogenous to $\Jac(C_a)^2$. Furthermore, one can check that $\Jac(C_a)$ is absolutely simple and has minimal endomorphism ring, so $\End(\Jac(X))=M_2(\ZZ)$.
\end{example}

\subsection{Hyperelliptic genus $4$ curves with Jacobians isogenous to a square}\label{S:M2_hyp}

If we insist that $X$ is hyperelliptic, we get the condition $\P(a,\sigma(a))=\P(\sigma(a),a)=0$. The expected dimension of this locus is $1$-dimensional, and it indeed contains several components of this dimension. However, as it turns out, for $\sigma=(1,2)(3,4)$, the equation $\P(a,\sigma(a))=0$ implies $P(\sigma(a), a)$ as well, so for that conjugacy class of $\sigma$, the curve $X$ is automatically hyperelliptic.

Recall from Figure~\ref{fig:V4} that a hyperelliptic $X$ is a double cover of both $C_a$ and $C_b$, with $C_a, C_b$ both covering the same $\PP^1$. Let us write $\alpha_1,\ldots,\alpha_6\in\PP^1(\kalg)$ for the support of the branch locus of $C_a$ over $\PP^1$ and similarly $\beta_1,\ldots,\beta_6$ for $C_b$. Let us assume that $\alpha_i=\beta_i$ for $i=1,\ldots, 5$; one has that $\alpha_6 \neq \beta_6$ when $X$ is irreducible.  The isomorphism $\Jac(C_a)[2] \simeq\Jac(C_b)[2]$ whose graph is the kernel of the isogeny $\Jac(C_a)\times\Jac(C_b)\to\Jac(X)$ is induced by the identification $\alpha_i\mapsto \beta_i$.

An isomorphism $C_a\to C_b$ induces an automorphism $\mu$ of $\PP^1$ that sends $\{\alpha_1,\ldots,\alpha_6\}$ to $\{\beta_1,\ldots,\beta_6\}$. It induces a permutation $\sigma\in S_6$ such that $\mu(\alpha_i)=\beta_{\sigma(i)}$.

We see that any cycle of $\sigma$ that does not include $6$ corresponds to a full orbit of $\mu$ of the same length, whereas the cycle that does contain $6$ corresponds to a partial orbit of $\mu$. The data that over an algebraically closed field $k$ specifies hyperelliptic curves $C,X$ of genera $2,4$ respectively with a $2$-isogeny $\Jac(C)^2\to\Jac(X)$ thus consists of a degree $6$ locus $B$ on $\PP^1$ together with an automorphism $\mu$ such that $B\cap\mu(B)$ has degree $5$.

The discrete data involved in this diagram involves the cycle type of $\sigma$, where we mark the cycle containing $6$. To indicate this we decorate it with a star: $6^*$. The singleton $(6^*)$ denotes the identity permutation, with the cycle containing $6$ marked. For our purposes the cycle types of  $(1,2,3,4)(5,6^*)$ and $(1,2)(3,4,5,6^*)$ are distinct.

The finite orbits of automorphisms of $\PP^1$ are rather restricted, which constrains the cycle types for $\sigma$ that can be realized. In particular:
\begin{itemize}
	\item If $\mu$ is not the identity then $\mu$ has one or two fixed points.
	\item If $\mu$ has a finite orbit of length larger than $1$, then $\mu$ has two fixed points and is of finite order and, outside of the fixed points, all orbits are of the same length.
\end{itemize}

\noindent
This rules out various cycle types.
For $\sigma=(6^*),(4,5)(6^*),(5,6^*),(4,5,6^*)$ we would need at least three fixed points and hence $\mu=\mathrm{id}$, which is inadmissible since then $\alpha_6$ is fixed too.
For $\sigma=(1,2)(4,5,6^*),(1,2)(3,4,5,6^*)$ we get that $\mu$ is of order $2$, so we cannot find an orbit of $\mu$ large enough to accommodate $\alpha_6$.
For $\sigma=(1,2)(5,6^*),(1,2)(3,4)(5,6^*),(1,2,3)(4,5,6^*)$ we get that the order of $\mu$ would be equal to the cycle length of $6^*$, which would make the cycle of $6^*$ a full orbit under $\mu$.

For the remaining cycle types we can choose coordinates, mostly by choosing $0,\infty$ as fixed points and setting one cycle to the orbit of $1$, but in some cases we get a nicer model by making different choices.
\begin{table}[t]
\[
\def\arraystretch{1.5}
\begin{array}{c|c|c}
	\text{type of }\sigma&f(x)&\mu\\
	\hline
	(1,2,3)(6^*)&x(x-1)(x^2-x+1)(x-u)&\mu(x)=\frac{1}{1-x}\\
	(1,2,3)(5,6^*)&(x^3-1)(x^2+ux+u^2)&\mu(x)=\zeta_3x\\
	(1,2)(3,4)(6^*)&(x^4+ux^2+v)(x-1)&\mu(x)=-x\\
	(1,2,3,4)(6^*)&(x^4-1)(x-u)&\mu(x)=\zeta_4x\\
	(3,4,5,6^*)&x(x-1)(x-u)(x-u^2)(x-u^3)&\mu(x)=ux\\
	(1,2,3,4)(5,6^*)&x(x-1)(x+1)(x-u)(x-\frac{u-1}{u+1})&\mu(x)=\frac{x-1}{x+1}\\
	(1,2,3,4,5)(6^*)&(x^5-1)(x-u)&\mu(x)=\zeta_5x\\
	(2,3,4,5,6^*)&(x-1)(x-u)(x-u^2)(x-u^3)(x-u^4)&\mu(x)=ux\\
	(1,2,3,4,5,6^*)&(x-1)(x-u)(x-u^2)(x-u^3)(x-u^4)(x-u^5)&\mu(x)=ux\\
\end{array}\]
\caption{Curves $C_a\colon y^2=f(x)$ and $C_b\colon y^2=f(\mu(x))$ leading to hyperelliptic $X$ with $M_2(\ZZ)\subset\End\Jac(X)$}\label{T:hyperellipticM2}
\end{table}
\begin{remark}
Note that for $\sigma=(2,3,4,5,6^*)$ we also have $f(x)=x(x-1)(x-2)(x-3)(x-4)$ and $\mu(x)=x+1$. With a change of model this is obtained from the limit $u\to 1$ for the listed family. A similar degeneration occurs for $\sigma=(1,2,3,4,5,6^*)$.
\end{remark}

Of particular interest is that $\sigma=(1,2)(3,4)(6^*)$ admits a $2$-parameter family. Indeed, $\P(a,\sigma(a))$ and $\P(\sigma(a),a)$ cut out the same locus. 

Peculiarly, for several of the other $\sigma$ there is an involution $\mu'$ of $\PP^1$ that yields $\sigma'$ of cycle type $(1,2)(3,4)(6^*)$. We tabulate them in Table~\ref{T:extra_mus}
\begin{table}[t]
\[
\def\arraystretch{1.5}
\begin{array}{c|c}
\text{type of }\sigma&\mu'(x)\\
\hline
(1,2,3)(6^*)& \frac{1}{x}\\
(1,2,3,4)(6^*)& \frac{1}{x}\\
(3,4,5,6^*)& \frac{u^2}{x}\\
(1,2,3,4,5)(6^*)& \frac{1}{x}\\
(2,3,4,5,6^*)& \frac{u^4}{x}\\
(1,2,3,4,5,6^*)& \frac{u^4}{x}\\
\end{array}\]
\caption{Cycle types with an additional automorphism $\mu'$ of type $(1,2)(3,4)(6^*)$}\label{T:extra_mus}
\end{table}
\begin{example}\label{E:2dim_hyplocus}
For $\sigma=(1,2)(3,4)(6^*)$ we have
\[C_{u,v}\colon y^2=(x^4+ux^2+v)(x+1).\]
The fibre product with the double cover $y_2^2=x^2-1$ yields a hyperelliptic curve of genus $4$ with a model
\[X\colon y^2=(x^2+1)(vx^8 + 4(u + v)x^6 + (8u + 6v + 16)x^4 + 4(u + v)x^2 + v).\]
We have that $\Jac(X)$ is $(2,2)$-isogenous to $\Jac(C_{u,v})^2$.
\end{example}

\bigskip

\renewcommand{\MR}[1]{}

\end{document}